\documentclass[11pt]{amsart}
\usepackage{graphics,latexsym,amssymb,amsmath}
\usepackage[dvips]{graphicx}
\usepackage{psfrag,pst-node}
\usepackage[german, english]{babel}
\input{amssym.def}

\newtheorem{thm}{Theorem}[section]

\newtheorem{lemma}[thm]{Lemma}

\newtheorem{cor}[thm]{Corollary}
\newtheorem{prop}[thm]{Proposition}

\theoremstyle{definition}
\newtheorem{rmk}[thm]{Remark}
\newtheorem{dfn}[thm]{Definition}
\newtheorem{example}{Example}

\newcommand{\C}{{{\mathbb C}}}
\newcommand{\R}{{{\mathbb R}}}
\newcommand{\Z}{{{\mathbb Z}}}
\newcommand{\N}{{{\mathbb N}}}
\newcommand{\Sp}{{{\mathbb S}}}

\newcommand{\A}{{{\mathcal A}}}

\newcommand{\al}{{\alpha}}

\newcommand{\de}{{\delta}}
\newcommand{\eps}{{\varepsilon}}
\newcommand{\ka}{{\kappa}}
\newcommand{\ph}{{\varphi}}
\newcommand{\lm}{{\lambda}}
\newcommand{\gm}{{\gamma}}

\newcommand{\ow}[1]{\widetilde{ #1}}

\newcommand{\PCPS}{poly\-gonal complex with planar substructures}
\newcommand{\PCPSs}{poly\-gonal complexes with planar substructures}
\newcommand{\defPCPS}{}

\newcommand{\PCSS}{polygonal complex with spherical substructures}

\newcommand{\defPCSS}{}

\newcommand{\PCPSSS}{polygonal complex with planar or spherical substructures}

\newcommand{\Hm}[1]{\leavevmode{\marginpar{\tiny%
$\hbox to 0mm{\hspace*{-0.5mm}$\leftarrow$\hss}%
\vcenter{\vrule depth 0.1mm height 0.1mm width \the\marginparwidth}%
\hbox to 0mm{\hss$\rightarrow$\hspace*{-0.5mm}}$\\\relax\raggedright #1}}}

\date{}
\begin{document}
\title[Polygonal complexes with planar substructures]{Sectional
  curvature of polygonal complexes with planar substructures}

\author[M.~Keller]{Matthias Keller}
\author[N.~Peyerimhoff]{Norbert Peyerimhoff}
\author[F.~Pogorzelski]{Felix Pogorzelski}

\address[M.~Keller, F.~Pogorzelski]{Friedrich-Schiller-Universit\"at Jena,
Fakult\"at f\"ur Mathematik und Informatik, Mathematisches Institut, Germany}
\email{m.keller@uni-jena.de}
\email{felix.pogorzelski@uni-jena.de}

\address[N.~Peyerimhoff]{Department of Mathematical Sciences, Durham
  University, Science Laboratories South Road, Durham, DH1 3LE, UK}
\email{norbert.peyerimhoff@durham.ac.uk}

\begin{abstract}
  In this paper we introduce a class of polygonal complexes for which
  we can define a notion of sectional combinatorial curvature. These
  complexes can be viewed as generalizations of $2$-dimensional
  Euclidean and hyperbolic buildings. We focus on the case of
  non-positive and negative combinatorial curvature. As geometric
  results we obtain a Hadamard-Cartan type theorem, thinness of
  bigons, Gromov hyperbolicity and estimates for the Cheeger
  constant. We employ the latter to get spectral estimates, show
  discreteness of the spectrum in the sense of a Donnelly-Li type theorem   and present corresponding eigenvalue asymptotics. Moreover, we prove a
  unique continuation theorem for eigenfunctions and the solvability
  of the Dirichlet problem at infinity.
\end{abstract}

\maketitle
\tableofcontents


\section{Introduction}
Since recent years there is an increasing interest in studying curvature notions on discrete spaces. First of all there are various approaches to Ricci curvature  based on $L^{1}$-optimal transport on metric measure spaces starting with the work of Ollivier, \cite{O1,O2}.  These ideas were employed for graphs by various authors \cite{BJL, JL, LLY, LY} to study geometric and spectral questions. A related and very effective definition using $L^{2}$-optimal transport was introduced in \cite{EM}. Secondly, in \cite{JL,LY} there is the approach  of defining curvature bounds via curvature-dimension-inequalities using a calculus of Bakry-Emery based on Bochner's formula for Riemannian manifolds. Similar ideas were used \cite{BHLLMY} to prove a Li-Yau inequality for graphs. Finally let us mention the work on so called Ricci-flat graphs \cite{CY} and \cite{Fo} for another approach.
All these approaches have in common that they model some kind of Ricci curvature and that they are very useful to study lower curvature bounds.

Classically there is a curvature notion for planar polygonal complexes, called tessellating graphs, defined by an angular defect. These ideas go back as far as to works of Descartes \cite{F} and became mathematical folklore since then. Often there is no obvious relation of this curvature to the recent notions of Ricci curvature above. Despite the rather restrictive setting of planar graphs this curvature notion has proven to very effective to derive very strong spectral and geometric consequences of  upper curvature bounds \cite{BP1,BP2,Hi,Kel,Kel2,KLPS,Woe} which often relate to results to upper bounds on sectional curvature of Riemannian manifolds. (For consequences on lower bounds see, e.g., \cite{DM,HJL,NS,St,Z} as well.) Thus, it seems desirable to identify a class of more general complexes where on can define and introduce sectional curvature. This is the aim of this work.

The objects under investigation in this article are {\em polygonal complexes with
  planar substructures}\defPCPS. They are $2$-dimensional CW-complexes
equipped with a family of subcomplexes homeomorphic to the Euclidean
plane, which we call {\em apartments}, since they have certain
properties similar to the ones required for apartments in Euclidean
and hyperbolic buildings. The $2$-cells of a {\PCPS} can be viewed as
polygons and they are called faces and their closures are called
chambers. The geometry is based on this set of faces and their
neighboring structures. In particular, there is a combinatorial
distance function on the set of faces. Let us discuss the properties
of apartments in more detail. First of all, we require that there are
enough apartments, that is any two faces have to lie in at least one
apartment (condition (PCPS1) in Definition~\ref{dfn:PCPS} below). Sometimes,
we require the stronger condition (${\rm PSPS1}^*)$ that every
infinite geodesic ray of faces is included in an apartment. The
second crucial property is that all apartments are {\em convex}
(see condition (PCPS2)). These properties are also similar to the ones
satisfied by flats in symmetric spaces. The definition of {\PCPSs} comprises
both planar tessellations and all $2$-dimensional Euclidean and hyperbolic buildings.

We use the apartments of a {\PCPS} to define combinatorial curvatures on
them. Since these apartments could be seen in a vague sense as tangent
planes of the {\PCPS}, we call these curvatures {\em sectional   curvatures}. We introduce sectional curvatures on the faces and on the corners of an apartment (see Definition \ref{dfn:curv}), and they are invariants measuring the {\em local geometry} of the {\PCPS}.

The definition of {\PCPSs} and basic notions are introduced in Section~\ref{s:basics}. The results in this article are then given in Sections~\ref{s:geometry} and \ref{s:spectrum}. While most of these results are known for planar tessellations, it seems to us that several of these results were not known for Euclidean and hyperbolic
buildings. Next, we explain our results in more detail.

In Section~\ref{s:geometry} we discuss implications of negative and non-positive curvature to the {\em global and asymptotic geometry} of
a {\PCPS}. Many of the presented results have well-known counterparts in the smooth setting of Riemannian manifolds. Amongst our results, we present a combinatorial
Cartan-Hadamard theorem for non-positively curved {\PCPSs} (see Theorem~\ref{t:Cut_locus}) and we conclude Gromov hyperbolicity and positivity of the
Cheeger isoperimetric constant for negatively curved {\PCPSs} with certain
bounds on the vertex and face degree (see Theorems~\ref{t:Gromov} and
\ref{t:cheeger1}). These results are based on negativity or non-positivity of the sectional corner curvature. We also state an analogue of Myers
theorem in the case of strictly positive sectional face curvature (see
Theorem \ref{t:myers}).

Section \ref{s:spectrum} is devoted to spectral considerations of the
Laplacian. We discuss combinatorial/geometric criteria to guarantee
emptiness of the essential spectrum and to derive certain eigenvalue
asymptotics on {\PCPSs} (see Theorem \ref{t:DiscreteSpectrum}). We also
show that non-positive sectional corner curvature on {\PCPSs} implies absence of finitely supported eigenfunctions (see Theorem
\ref{t:uniqueContinuation}). Finally, we derive solvability of the Dirichlet problem at infinity for {\PCPSs} in the case of negative sectional corner curvature (see
Theorem \ref{thm:Dirichlet}).

As mentioned before, $2$-dimensional Euclidean and hyperbolic buildings provide large classes of examples of {\PCPSs}. While all these spaces have non-positive sectional face curvature, their corner curvature is not always necessarily non-positively curved. The main purpose of the final
Section \ref{s:examples} is to provide a self-contained short survey over these important classes.

\bigskip

{\bf Acknowledgements:} We like to thank Oliver Baues, Shiping Liu, Alex Lubotzky, Shahar Mozes, Dirk Sch\"utz and Alina Vdovina for many useful discussions. This
research was partially supported by the EPSRC Grant EP/K016687/1 (N.P.) and by the DFG (M.K.). Part of the work was done at the LMS-EPSRC Durham Symposium ``Graph Theory and Interactions.''
F.P. gratefully acknowledges the support from the German National Academic Foundation (Studienstiftung des deutschen Volkes). \\
This work would not be the same without the cheerful atmosphere created by
Lumen Keller during the initial stage of this collaboration.


\section{Basic definitions} \label{s:basics}

In this section we introduce {\em polygonal complexes with planar
  substructures} \defPCPS and define a notion of sectional curvature on
theses spaces. In order to do so we introduce polygonal complexes and
planar tessellations first. In the second subsection we explore some
basic consequences of the convexity assumption we impose. In the third
subsection we introduce a combinatorial sectional curvature notions for
these spaces.

\subsection{Polygonal complexes with planar
  substructures\defPCPS} \label{s:PCPS}

The following definition of polygonal complexes is found in \cite{BB1}.

\begin{dfn}[Polygonal complex] \label{dfn:polycomp}
  A \emph{polygonal complex} is a 2-dimensional CW-complex $X$ with
  the following properties:
  \begin{itemize}
  \item [(PC1)] The attaching maps of $X$ are homeomorphisms.
  \item [(PC2)] The intersection of any two closed cells of $X$ is
    either empty or exactly one closed cell.
  \end{itemize}
\end{dfn}

For a polygonal complex $X$ we denote the set of $0$-cells by $V$ and
call them \emph{vertices}, we denote the set of $1$-cells by $E$ and
call them the \emph{edges} and we denote the set of $2$-cells by $F$
and call them the \emph{faces}. We write $X=(V,E,F)$. Note that the
closures of all edges and faces in $X$ are necessarily compact (since
they are images of compact sets under the continuous characteristic
maps, see \cite[Appendix]{Hat}). We call two vertices $v$ and $w$
\emph{adjacent} or \emph{neighbors} if they are connected by an edge
in which case we write $v\sim w$. We call two different faces $f$ and
$g$ \emph{adjacent} or \emph{neighbors} if their closures intersect in
an edge and we write $f \sim g$. It is convenient to call the closure
of a face a {\em chamber}.

The \emph{degree} $|v| \in {\mathbb N}_0 \cup \{ \infty \}$ of a
vertex $v\in V$ is the number of vertices that are adjacent to
$v$. The \emph{degree} $|e| \in {\mathbb N}_0 \cup \{ \infty \}$ of an
edge $e\in E$ is the number of chambers containing $e$. The boundary
$\partial f$ of a face $f \in F$ is the set of all $1$-cells $e \in E$
being contained in the closure $\overline{f}$. Since in CW-complexes
every compact set can meet only finitely many cells (see
\cite[Prop. A.1]{Hat}), we have $| \partial f | = \# \partial f  < \infty$. The
\emph{degree} $|f|$ of a face $f\in F$ is the number of faces that are
adjacent to $f$ and, in contrast to $| \partial f |$, the face degree
$|f|$ can be infinite.

We call a (finite, infinite or bi-infinite) sequence
$\dots,f_{i-1},f_i,f_{i+1},\dots$ of pairwise distinct faces a
\emph{path} if successive faces are adjacent. The \emph{length} of the
path is one less than the number of components of the sequence. The
\emph{(combinatorial) distance} between two faces $f$ and $g$ is the length of the
shortest path connecting $f$ and $g$ and the distance is denoted by
$d(f,g)$. We call a (finite, infinite or bi-infinite) path
$(f_{k})$ of faces a \emph{geodesic} or a \emph{gallery}, if we
have for any two faces $f_{m}$ and $f_{n}$ in the path
$d(f_{m},f_{n})=|m-n|$, i.e., the distance between $f_{m}$ and $f_{n}$
is realized by the path.

We say a polygonal complex $X$ is \emph{planar} if $X$ is homeomorphic
to ${\mathbb R}^2$. We also say that a polygonal complex $X$ is
\emph{spherical} if $X$ is homeomorphic to the two-sphere ${\mathbb
  S}^2$.

Next we introduce the notion of a planar tessellation following
\cite{BP1,BP2}.

\begin{dfn}[Planar tessellation] \label{dfn:plantess}
  A polygonal complex $\Sigma=(V,E,F)$ is called a
  \emph{(planar/spherical) tessellation} if $\Sigma$ is
  planar/spherical and satisfies the following properties:
  \begin{itemize}
  \item [(T1)] Any edge is contained in precisely two different
    chambers.
  \item [(T2)] Any two different chambers are disjoint or have
    precisely either a vertex or a side in common.
  \item [(T3)] For any chamber the edges contained in it form a closed
    path without repeated vertices.
  \item [(T4)] Every vertex has finitely many neighbors.
  \end{itemize}
\end{dfn}

Note that property (T3) is already implied by (PC1) and (PC2). The
tessellations form the substructures which we will later need to
define sectional curvature. Now, we are in a position to introduce
polygonal complexes with planar substructures.

\begin{dfn} \label{dfn:PCPS}
  A \emph{polygonal complex with planar substructures} \defPCPS is a
  polygonal complex $X=(V,E,F)$, together with a set $\mathcal{A}$ of
  subcomplexes whose elements
  $\Sigma=(V_{\Sigma},E_{\Sigma},F_{\Sigma})$ are called
  \emph{apartments}, with the following properties:
  \begin{itemize}
  \item [(PCPS1)] For any two faces there is an apartment containing
    both of them.
  \item [(PCPS2)] The apartments are convex (i.e., for any $\Sigma \in
    \mathcal{A}$ any finite gallery $f_1,\dots,f_n$ with end-faces
    $f_1,f_n$ in $\Sigma$ stays completely in $\Sigma$).
  \item [(PCPS3)] The apartments are planar tessellations.
  \end{itemize}
\end{dfn}

Similarly, we introduce \emph{polygonal complexes with spherical
  substructures} \defPCSS by replacing property (PCPS3) in Definition
\ref{dfn:PCPS} by

\begin{itemize}
\item[(PCSS3)] The apartments are spherical tessellations.
\end{itemize}

Prominent examples of {\PCPSs} are $2$-dimensional Euclidean and hyperbolic buildings (see
Section \ref{s:examples} for the definition of a building as well as
several examples). Moreover, every planar tessellation is trivially a
{\PCPS}. For reasons of illustration, we like to introduce the following
example of a Euclidean building.

\begin{example} \label{ex:vdov} Let $X_0$ be the finite simplicial
  complex constructed from the  seven equilateral Euclidean
  triangles illustrated in Figure~\ref{f:triangle} by identifying sides with the same labels $x_i$.

\begin{figure}[h]
  \begin{center}
      \psfrag{x0}{$x_0$}
      \psfrag{x1}{$x_1$}
      \psfrag{x2}{$x_2$}
      \psfrag{x3}{$x_3$}
      \psfrag{x4}{$x_4$}
      \psfrag{x5}{$x_5$}
      \psfrag{x6}{$x_6$}
     \includegraphics[width=\textwidth]{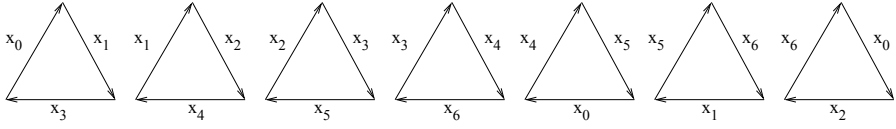}
  \end{center}
   \caption{Labeling scheme for the simplicial complex $X_0$}
   \label{f:triangle}
\end{figure}

Then $X_0$ has a single vertex which we denote by $p_0$, seven edges
and seven faces. Its fundamental group $\Gamma = \pi_1(\Pi_0,p_0)$ has
the following presentation
$$ \Gamma = \langle x_0,\dots,x_6 \mid x_i x_{i+1} x_{i+3} = {\rm id} \;
\text{for} \; i=0,1,\dots,6 \rangle $$ (where $i$ is taken modulo
$7$). Let $X = (V,E,F)$ be the universal covering of $X_0$ together
with the lifted triangulation. Then it follows from \cite[Theorem~
6.5]{BB2} that $X$ is a {\em thick Euclidean building of type
  $\widetilde A_2$} and every edge of $X$ belongs to precisely $3$
triangles. Therefore, $X$ is a {\PCPS}. The group of covering
transformations is isomorphic to $\Gamma$ and acts transitively on the
vertices of this building (see \cite{CMS}).
\end{example}

For some of our results we need the following slightly stronger
assumption than (PCPS1):

\begin{itemize}
\item [(${\rm PCPS1}^*$)] Every (one-sided) infinite geodesic is included in
  an apartment.
\end{itemize}

Condition (${\rm PCPS1}^*$) is satisfied for all $2$-dimensional Euclidean and
hyperbolic buildings with a {\em maximal apartment system} (see
Theorem \ref{thm:maxbuildPCPS} below).

Finally, let us mention the following important fact. To a polygonal
complex $X=(V,E,F)$ we can naturally associate a graph $G_X$ by letting $F$ be the vertex set of $G_X$ and by
defining the edges of that graph via the adjacency relation of the
corresponding faces. This 'duality' becomes important when we use
results for graphs in our context.


\subsection{Consequences of convexity}

The convexity assumption (PCPS2) is very important in our
considerations. In this subsection we collect some of the immediate
consequences.

\begin{lemma}\label{l:d} Let $X$ be a {\PCPS},  $\Sigma$ an apartment and let $d_{\Sigma}$ the combinatorial distance within the apartment.
  Then, for any two faces $f,g \in F_{\Sigma}$
  $$
  d(f,g)=d_{\Sigma}(f,g).
  $$
\end{lemma}

\begin{proof} The inequality '$\le$' is clear. For the other direction
  '$\ge$' let $\gm=(f_{0},\ldots,f_{n})$ be a path connecting $f$ and
  $g$ minimizing $d(f,g)$. As $\gm$ is a geodesic with end-faces in
  $\Sigma$ it is completely contained in $\Sigma$ by (PCPS2). Hence,
  the statement follows.
\end{proof}

We say a subset $F_{0}$ of $F$ is connected if any two faces in
$F_{0}$ can be joined by a path in $F_{0}$.

\begin{lemma} \label{l:intersection}
  Let $X$ be a {\PCPS}. Let $\Sigma_{1}$ and $\Sigma_{2}$ be two
  apartments of $X$. Then the set $F_{\Sigma_{1}}\cap F_{\Sigma_{2}}$
  is connected.
\end{lemma}

\begin{proof} Let $f$ and $g$ be two faces in $F_{\Sigma_{1}} \cap
  F_{\Sigma_{2}}$. Then, by (PCPS2), every geodesic connecting $f$ and
  $g$ is completely contained in $\Sigma_{1}$ and $\Sigma_{2}$. Thus,
  $F_{\Sigma_{1}} \cap F_{\Sigma_{2}}$ is connected.
\end{proof}

For a fixed face $o \in F$ (called {\em center}), we define the
(combinatorial) spheres and balls about $o$ by
\begin{eqnarray*}
    S_{n} &=& S_{n}(o) = \{ f \in F\mid d(f,o)=n \} \quad \mbox{and} \\
    B_{n} &=& B_{n}(o) = \bigcup_{k=0}^{n} S_{k},
\end{eqnarray*}
for $n \ge 0$. For $f \in F$, we let the \emph{forward} and
\emph{backward degree} be given by
\begin{align*}
  |f|_{\pm} = | \{ g \in F \mid g \sim f,\, d(g,o) = d(f,o) \pm 1 \} |,
\end{align*}
and we call $g \in F$ with $g \sim f$ and $d(g,o) = d(f,o) + 1$
(respectively $d(g,o) = d(f,o) - 1$) a \emph{forward} (respectively
\emph{backward}) \emph{neighbor} of $f$. The next lemma shows that the
convexity condition (PCPS2) imposes a lot of structure of the distance spheres.

\begin{lemma} \label{l:neighbors} Let $X$ be a {\PCPS} and $o \in F$ be a
  center. Let $f \in F$ with $f\in S_{n}$ for some $n \ge 0$ and $f_{+}
  \in S_{n+1}$, $f_{0} \in S_{n}$, $f_{-} \in S_{n-1}$ be neighbors of
  $f$. Then,
  \begin{itemize}
  \item [(a)] Every face sharing the same edge with $f$ and $f_{+}$ is
    in $S_{n+1}$.
  \item [(b)] Every face sharing the same edge with $f$ and
    $f_{0}$ is in $S_{n}\cup S_{n-1}$.
  \item [(c)] Every face sharing the same edge with $f$ and
    $f_{-}$ is in $S_{n}$.
  \end{itemize}
\end{lemma}

\begin{proof}(a) Let $g \in F$ be such that $\partial g \cap \partial
  f \cap \partial f^+ \neq \emptyset$. Since $g$ is a neighbor of
  $f_{+}$, we have $d(o,f) \ge n$. Since $g$ is a neighbor of $f$, we
  have $d(o,f) \le n+1$. Therefore, we have $g \in S_{n}\cup
  S_{n+1}$. If $g$ was in $S_{n}$, then there are geodesics from the
  center $o$ over $g$ to $f_{+}$ and from $o$ over $f$ to $f_{+}$. By
  (PCPS2) both of these geodesics lie together in one
  apartment. Hence, $g$ lies in one apartment together with $f$,
  $f_{+}$ and $o$. Then, there is an edge contained in three faces
  $f,f_{+}$ and $g$ within one apartment $\Sigma$. This contradicts
  (T1) in the definition of a planar tessellation. But $\Sigma$ is a
  planar tessellation, by (PCPS3). Thus, $g\in S_{n+1}$.\\
  (b) Let $g \in F$ be such that $\partial g \cap \partial f
  \cap \partial f_{0} \neq \emptyset$. If $g$ was in $S_{n+1}$ then
  there were two geodesics from $o$ to $g$, one via $f$ and the other
  one via $f_{0}$. By a similar argument as in (a), the faces $g$,
  $f$, $f_{0}$ and $o$ lie in the same apartment. Again this is
  impossible by (T1) and (PCPS3).\\
  (c) Let $g \in F$ be such that $\partial g \cap \partial f
  \cap \partial f_{-} \neq \emptyset$. Clearly, $g$ is in $S_{n}\cup
  S_{n-1}$. If $g$ was in $S_{n}$ then, by similar arguments as in (a)
  and (b), the faces $g$, $f$, $f_{-}$ and $o$ lie in the same
  apartment which is again impossible by (T1) and (PCPS3).
\end{proof}


\subsection{Sectional curvature}

For an apartment $\Sigma=(V_{\Sigma},E_{\Sigma},F_{\Sigma})$, let
$|v|_{\Sigma}$ be the \emph{degree of $v$ in $\Sigma$} which is the
number of neighboring vertices in $V_{\Sigma}$. We notice that the
degree of an edge in $\Sigma$, i.e., the number of faces in
$F_{\Sigma}$ bounded by the edge, is always equal to $2$ by
(T1). Moreover, the \emph{degree} ${|f|}_{\Sigma}$ \emph{of a face}
$f$ \emph{in $\Sigma$} is equal to $| \partial f |$. Therefore,
$|f|_{\Sigma_{1}}=|f|_{\Sigma_{2}}$ for any two apartments
$\Sigma_{1}$, $\Sigma_{2}$ that contain $f$. Furthermore, for a {\PCPS}
$X$ and $\Sigma\in \A$ we let the set of corners of $X$ and of
$\Sigma$ be given by
\begin{align*}
    C = \{ (v,f) \in V \times F \mid v \in f \}, \quad
    C_\Sigma = \{ (v,f) \in V_{\Sigma} \times F_{\Sigma} \mid v \in f \}.
\end{align*}

\begin{dfn}[Sectional Curvature] \label{dfn:curv}
  Let $\Sigma$ be an apartment of a polygonal complex with planar
  substructures $X$. The \emph{sectional corner curvature}
  $\ka_{c}^{(\Sigma)}: C_\Sigma \to \R$ with respect to
  $\Sigma$ is given by
  \begin{align*}
    \ka_{c}^{(\Sigma)}(v,f) = \frac{1}{|v|_{\Sigma}} -
    \frac{1}{2} + \frac{1}{{|f|}_{\Sigma}},
  \end{align*}
  and the \emph{sectional face curvature} $\ka^{(\Sigma)}: F_\Sigma \to \R$
  with respect to $\Sigma$ is given as
  \begin{align*}
    \ka^{(\Sigma)}(f) = \sum_{(v,f) \in C_{\Sigma}} \ka_{c}^{(\Sigma)}(v,f) =
    1 - \frac{{|f|}_{\Sigma}}{2} + \sum_{v \in V_{\Sigma}, v \in \overline{f}}
    \frac{1}{|v|_{\Sigma}}.
  \end{align*}
\end{dfn}

The above combinatorial curvature notions are motivated by a
combinatorial version of the Gau\ss-Bonnet Theorem for closed
surfaces. We have for polygonal tessellations $\Sigma=
(V,E,F)$ of a closed surface $S$ (see \cite[Theorem~1.4]{BP1})
\begin{equation*} \label{eq:GB}
\chi(S) = \sum_{f \in F_{\Sigma}} \ka^{(\Sigma)}(f) 
\end{equation*}
where $\chi(S)$ is the Euler characteristic of $S$.
The sectional curvatures in Definition \ref{dfn:curv} are then the
intrinsic curvatures in the apartments $\Sigma$, and the apartments
$\Sigma$ can be understood as discrete analogues of specific tangent
planes. Note that curvature is a local concept and, for a given corner
or face, only information of the nearest neighboring faces in the
apartment are needed for its calculation.

Let us finally mention that the apartments in Example \ref{ex:vdov} are
regular tessellations of a Euclidean plane by equilateral triangles and that
this example has vanishing sectional face and corner curvature. This is a special
case covered by Proposition \ref{prop:curvEuclid} in Section~\ref{s:Euclidean buildings}, which presents curvature properties in the general situation of Euclidean buildings.



\section{Geometry} \label{s:geometry}

In this section we discuss implications of the curvature sign to the
global geometry of {\PCPSs}  like emptiness of cut-locus, Gromov
hyperbolicity and positivity of the Cheeger constant. Before we enter
into these topics, we first introduce some more useful combinatorial
notions.

We say $X$ is \emph{locally finite} if for all $v \in V$ and $e \in E$
\begin{align*}
    |v|<\infty\quad\mbox{and}\quad |e|<\infty.
\end{align*}
Since $|f| = \sum_{e \in \partial f} (|e|-1)$, we also have $|f| < \infty$ for locally finite polygonal complexes. For locally finite $X$, we define for a face $f \in F$
\begin{align*}
  m_{E}(f)&=\min_{e \in \partial f} (|e|-1),\qquad M_{E}(f)=\max_{e
    \in \partial f}(|e|-1)
\end{align*}
the minimal and maximal number of neighbors over one edge of $f$. The
\emph{minimal} and \emph{maximal thickness} of $X$ is then given by
\begin{align*}
  m_{E} &= \min_{f\in F} m_{E}(f),\quad
  M_{E} =\sup_{f\in F} M_{E}(f).
\end{align*}
The {\em maximal vertex} and {\em face degree} are defined by
\begin{align*}
  M_{V} = \sup_{v\in V}|v|, \qquad
  M_{F} = \sup_{f\in F}|f|.
\end{align*}
Note that we always have $M_{E} \le M_{F}$ and both can be infinite.

\subsection{Absence of cut-locus}

We first present a theorem which is an analogue of the Hadamard-Cartan
theorem from Riemannian manifolds. It is a rather immediate
consequence of convexity and \cite[Theorem~1]{BP2} for plane tessellating
graphs.

For a face $f \in F$ in a polygonal complex $X=(V,E,F)$ the \emph{cut
  locus} of $f$ is defined as
\begin{align*}
  \mathrm{Cut}(f)=\{g\in F\mid d(f,\cdot) \mbox{ attains a local
    maximum in }g\}.
\end{align*}
{\em Absence of cut locus} means that $\mathrm{Cut}(f)=\emptyset$ for
all $f\in F$ which means that every finite geodesic starting in $f$
can be continued to a infinite geodesic.

\begin{thm} \label{t:Cut_locus} Let $X=(V,E,F)$ be a {\PCPS} such that
  $\ka_{c}^{(\Sigma)} \leq 0$ for all apartments $\Sigma \in {\mathcal
    A}$. Then, $\mathrm{Cut}(f)=\emptyset$ for all $f\in F$. Moreover,
  every geodesic within an apartment $\Sigma$ can be continued to an
  infinite geodesic within $\Sigma$.
\end{thm}

We conclude from Theorem \ref{t:Cut_locus} that emptiness of cut-locus
holds, e.g., for our Example \ref{ex:vdov} and Examples
\ref{ex:prodtrees}-\ref{ex:hag} (found in Section
\ref{s:examples}). Note also that the condition of non-positive
sectional {\em corner curvature} in Theorem \ref{t:Cut_locus} cannot
be weakened to non-positive sectional {\em face curvature} as Figure 2
in \cite{BP2} shows.

\begin{proof} Let $f\in F$. Choose $g\in F$ and let $\Sigma$ be an   apartment which contains both $f$ and $g$ (which exists by   (PCPS1)). By \cite[Theorem~1]{BP2} the cut locus of $f$ within   $\Sigma$ is empty that is there is a face $h\in F_{\Sigma}$ with $g\sim   h$ such that $d_{\Sigma}(f,h)=d_{\Sigma}(f,g)+1$. (Note that
  \cite[Theorem~1]{BP2} is formulated in the dual setting which,   however, can be carried over directly.) As $d=d_{\Sigma}$ on $\Sigma$, by   Lemma~\ref{l:d}, we conclude $g\not\in \mathrm{Cut}(f)$. Since this
  holds for all $g\in F$, we have  $\mathrm{Cut}(f)=\emptyset$. The
  second statement is an immediate consequence of
  \cite[Theorem~1]{BP2} and Lemma~\ref{l:d}.
\end{proof}

\begin{cor} Let $X=(V,E,F)$ be a {\PCPS} such that $\ka_{c}^{(\Sigma)}
  \leq 0$ for all $\Sigma \in {\mathcal A}$. Then, every face has at
  least one forward neighbor.
\end{cor}


\subsection{Thinness of bigons}

In this subsection we show a useful hyperbolicity criterion.

Let $X=(V,E,F)$ be a polygonal complex. A \emph{bigon} is a pair of
geodesics $(f_{0},\ldots,f_{n})$ and $(g_{0},\ldots,g_{n})$ such that
$f_{0}=g_{0}$ and $f_{n}=g_{n}$. We say a bigon is {\em $\de$-thin}
for $\de \ge 0$, if $d(f_{k},g_{k}) \leq \de$ for all $k=0,\ldots,n$.

\begin{thm} \label{t:bigons} Let $X=(V,E,F)$ be a {\PCPS} such that
  $\ka_{c}^{(\Sigma)} < 0$ for all apartments $\Sigma \in {\mathcal
    A}$. Then, every bigon is $1$-thin.
\end{thm}

\begin{proof} Let $\gm_{1}=(f_{0},\ldots,f_{n})$ and
  $\gm_{2}=(g_{0},\ldots,g_{n})$ be a bigon and $\Sigma \in {\mathcal
    A}$ be an apartment that contains $f_{0}=g_{0}$ and
  $f_{n}=g_{n}$. By the convexity assumption (PCPS2) the apartment
  $\Sigma$ contains both geodesics $\gm_{1}$ and $\gm_{2}$ and,
  therefore, the pair $(\gm_{1},\gm_{2})$ is a bigon within
  $\Sigma$. By \cite[Theorem~2]{BP2} it follows that
  $d_{\Sigma}(f_{k},g_{k}) \leq 1$ for $k=0,\ldots,n$, and by
  Lemma~\ref{l:d} we conclude that $d(f_{k},g_{k}) \leq 1$ for
  $k=0,\ldots,n$.
\end{proof}

We have an immediate consequence.

\begin{cor} \label{c:bigon} Let $X=(V,E,F)$ be a {\PCPS} such that
  $\ka_{c}^{(\Sigma)} < 0$ for all $\Sigma \in {\mathcal A}$. Let
  $f_1,f_2 \in F$ with $d(f_1,f_2) = n$. Then we have for all
  $0 \le k \le n$:
  $$ |B_k(f_1) \cap B_{n-k}(f_2)| \le 2. $$
  In particular, if $f_1$ is considered as a center, $f_2$ has at most
  two backward neighbors.
\end{cor}

\begin{proof}
  By convexity we can restrict our considerations on any apartment
  $\Sigma \in {\mathcal A}$ containing $f_1$ and $f_2$. Every $f \in
  B_k(f_1) \cap B_{n-k})(f_2)$ must obviously satisfy $d(f_1,f) = k$
  and $d(f,f_2)=n-k$. If there were $3$ faces in the intersection
  $B_k(f_1) \cap B_{n-k}(f_2) \subset F_\Sigma$, then there are $3$
  geodesics from $f_1$ to $f_2$ in $\Sigma$. Then, one of the three
  geodesics is enclosed by the other two in $\Sigma$ and the other two
  geodesics form a bigon. Then this bigon in not $1$-thin which
  contradicts the previous theorem.
\end{proof}

In fact, using the techniques of \cite{BP2} the last statement of Corollary \ref{c:bigon} holds even for
non-positive sectional corner curvature.

\begin{prop} \label{prop:bigon} Let $X=(V,E,F)$ be a {\PCPS} such that
  $\ka_{c}^{(\Sigma)} \le 0$ for all $\Sigma \in {\mathcal A}$ and $o
  \in F$ be a center. Then every face has at most two backward
  neighbors.
\end{prop}

\begin{proof}
  This is a consequence of the results in \cite{BP2}. Let $f \in F$.
  Let $\Sigma \in {\mathcal A}$ be an apartment containing $o$ and
  $f$. Then the ball $B_n \cap \Sigma$ is an admissible polygon in
  $\Sigma$ in the sense of \cite[Def. 2.2]{BP2} and $\partial f\cap\partial B_n$ is a connected path of length $\le 2$, by
  \cite[Prop. 2.5]{BP2}. This shows that $f$ can have at most two
  backward neighbors.
\end{proof}


\subsection{Gromov hyperbolicity}

Recall from the end of Subsection \ref{s:PCPS} that every polygonal
complex $X=(V,E,F)$ can also be viewed as a metric space via the
associated graph $G_X$ and its natural combinatorial
distance function. Geodesics $(f_i) \subset F$ in $X$ correspond
then to (vertex) geodesics in $G_X$. With this understanding, we
call the polygonal complex $(X,d)$ Gromov hyperbolic if there exists
$\de > 0$ such that any side of any geodesic triangle in $G_X$ lies in the $\delta$-neighborhood of the union of the two
other sides of the triangle. We show Gromov hyperbolicity of a {\PCPS}
$(X,d)$ with negative sectional corner curvature as well as properties
of the Gromov boundary $X(\infty)$ under the additional boundedness
assumption of the vertex and face degree. For details on the Gromov
boundary (and the Gromov product used to define it) we refer to
\cite[Chpt. III.H]{BH}.

\begin{thm} \label{t:Gromov} Let $X$ be a {\PCPS} with $M_{V}, M_{F} <
  \infty$ and $\ka_{c}^{(\Sigma)} < 0$ for all $\Sigma \in {\mathcal
    A}$. Then, $(X,d)$ and all its apartments are Gromov hyperbolic
  spaces. If additionally (${\rm PCPS1}^*$) is satisfied then every
  connected component of the Gromov boundary $X(\infty)$ contains the
  Gromov boundary of an apartment which is homeomorphic to the unit circle $S^1$.
\end{thm}

By the theorem in the section above all bigons in $(X,d)$ are
$1$-thin. For Cayley graphs, \cite[Theorem~1.4]{Pa} tells us that the
statement of the theorem above is true. For general $G_X$, we need the
following generalization given in the unpublished Master's
dissertation of Pomroy (a proof of it can be found in
\cite[Appendix]{ChN}). Here, a \emph{$\rho$-bigon} in a geodesic
metric space with metric $d$ is a pair of $(1,\rho)$ quasi-geodesics
$\gamma_1,\gamma_2$ with the same end points, i.e.,
$$ |t-t'| - \rho \le d(\gamma_i(t),\gamma_i(t')) \le |t-t'| + \rho,  $$
for all $ t,t'$.

\begin{thm}[Pomroy] \label{thm:pomroy}
  If for a geodesic metric space there are $\eps,\rho > 0$ such   that $\rho$-bigons are uniformly $\eps$-thin, then the space is
  Gromov hyperbolic.
\end{thm}

\begin{proof}[Proof of Theorem~\ref{t:Gromov}]
By Theorem~\ref{t:bigons} all bigons in $(X,d)$ are $1$-thin. The same holds true within all
apartments. For $G_X$ to satisfy the requirement of a {\em geodesic} metric space, we view it as a {\em metric graph} with all its edge lengths equal to one.
Choose $\rho < 1/2$ and $\eps=1$, we can then conclude from
Theorems~\ref{t:bigons} and \ref{thm:pomroy} that $(X,d)$ and all its
apartments are Gromov hyperbolic.

Next we prove the rest of the theorem assuming (${\rm PCPS1}^*$). From
$M_{F} < \infty$ we conclude that $G_X$ is a {\em proper} (i.e., closed balls in $G_{X}$ of finite radius are compact)
hyperbolic geodesic space and, therefore, the geodesic boundary
(defined via equivalence classes of geodesic rays, where rays are
equivalent iff they stay in bounded distance to each other) and the
Gromov boundary coincide (see, e.g., \cite[Lm. III.H.3.1]{BH}) and we
can think of any boundary point $\xi \in X(\infty)$ as being
represented by a geodesic ray $(f_i)$ of faces in $ F$. Using (${\rm
  PCPS1}^*$), there is an apartment $\Sigma \in {\mathcal A}$ such
that all the faces $f_i$ are in $ F_\Sigma$ and $\xi \in \Sigma(\infty) \subset
X(\infty)$. We also know from \cite[Cor. 5]{BP2} that $\Sigma(\infty)$
is homeomorphic to $S^1$, finishing the proof.

\end{proof}

It is easy to see that the Euclidean buildings in Example~\ref{ex:vdov} and \ref{ex:prodtrees} are not Gromov
hyperbolic. Theorem~\ref{t:Gromov} is not applicable since these examples have vanishing sectional corner curvature.


\subsection{Cheeger isoperimetric constants}

In this subsection we prove how negative curvature   effect positivity of the Cheeger isoperimetric constant.

Let $X=(V,E,F)$ be a {\em locally finite} polygonal complex. We
consider the following Cheeger constant which is very useful for
spectral estimates. For $H \subseteq F$, we define
\begin{align*}
  \al_{H}=\inf_{K\subseteq H\,\mbox{{\scriptsize finite}}}\frac{|\partial K|}{\mathrm{vol}(K)}
\end{align*}
with
\begin{align*}
\partial K=\{(f,g) \in K \times F\setminus K \mid f\sim g\}
\end{align*}
and
\begin{align*}
  \mathrm{vol}(K)=\sum_{f\in K} |f|.
\end{align*}
Note that $\alpha_H \le 1$. We set $ \al=\al_{F}$.

Firstly, we present a result that shows positivity of the Cheeger
isoperimetric constant for negative sectional corner curvature under
the additional assumption of bounded geometry. This result is a
consequence of a general result of Cao \cite{C}, which also holds in
the smooth setting of Riemannian manifolds. Secondly, we give more
explicit estimates for the Cheeger constant.

\begin{thm} \label{t:cheeger1} Let $X=(V,E,F)$ be a {\PCPS} such that
  $\ka_{c}^{(\Sigma)}<0$ for all $\Sigma \in {\mathcal A}$. Assume
  that $X$ additionally satisfies (${\it PCPS1}^*$) and $M_{V},
  M_{F}<\infty$. Then, $\al>0$.
\end{thm}

A straightforward consequence of Theorem \ref{t:cheeger1} and Theorem~\ref{thm:maxbuildPCPS} is the following result.

\begin{cor} Every $2$-dimensional locally finite hyperbolic building with regular
  hyperbolic polygons as faces has a positive Cheeger constant $\al > 0$.
\end{cor}
\begin{proof}Note that negative sectional curvature and the definition do not depend on the choice of the apartment systems. Hence, we switch to the corresponding building with maximal apartment system to obtain (${\it PCPS1}^*$) by Theorem~\ref{thm:maxbuildPCPS}. We conclude the statement by Theorem \ref{t:cheeger1}.
\end{proof}

In particular, all buildings in Examples \ref{ex:bourdon}-\ref{ex:hag}
have positive Cheeger constant.

\begin{proof}[Proof of Theorem \ref{t:cheeger1}]
  Note that by the comment at the end of Subsection~\ref{s:PCPS} we
  can associate to every {\PCPS} $X=(V,E,F)$ a graph $G_X$ by
  considering the faces of $X$ as vertices in $G_X$ and the
  edge relation given by the adjacency relation of the faces. In this
  light \cite[Theorem~1]{C} tells us that a polygonal complex $(X,d)$
  has positive Cheeger isoperimetric constant if the following four
  assumptions are satisfied
  \begin{itemize}
    \item [(1)] $(X,d)$ has bounded face degree $M_{F} < \infty$,
    \item [(2)]  $(X,d)$ admits a quasi-pole,
    \item [(3)]  $(X,d)$ is Gromov hyperbolic,
    \item [(4)] every connected component of the Gromov boundary
      $X(\infty)$ has positive diameter (with respect to a fixed
      Gromov metric),
  \end{itemize}
  where (2) means that there is a finite set $\Omega \subset F$ of
  faces and a $\de>0$ such that every face $f \in F$ is found in a
  $\de$-neighborhood of a geodesic ray emanating from this finite
  set. Moreover, for (4) we follow \cite{C} and define for two
  geodesic rays $(f_i), (f_i') \subset F$ with the same initial face
  $f_0=f_0'$ representing the points $\xi,\eta \in X(\infty)$:
  $$ d_{f_0,\epsilon}(\xi,\eta) =
  \liminf_{n \to \infty} \exp(-\eps(n-\tfrac{1}{2}d(f_n,f_n')), $$
  Then there is an $\eps > 0$ such that $d_{f_0,\epsilon}$ is a metric
  which is called a Gromov metric. Note that the Cheeger constant
  considered in \cite{C} is defined as
  $$ h = \inf_{H\subseteq F}\frac{|\partial_{F} H|}{|H|}, $$
  where $\partial_{F} H=\{f\in F\mid d(f,H)=1\}$. As every face in
  $\partial_{F}H$ is connected with $H$ via at least one edge we have
  $|\partial H| \ge |\partial_{F} H|$. Also $\mathrm{vol}(H)\le
  M_{F}|H|$ and, therefore,
  \begin{align*}
    \al \ge \frac{h}{M_{F}}.
  \end{align*}
  Hence, by the assumption $M_{F} < \infty$ the constant $\al$ is
  positive whenever $h$ is. Thus, it remains to check the conditions
  (1)-(4).

  \smallskip

  Let $X=(V,E,F)$ be a {\PCPS} which satisfies the assumptions of the
  theorem. Then, (1) is obviously satisfied. Secondly, by absence of
  cut-locus, Theorem~\ref{t:Cut_locus}, condition (2) is satisfied and
  by Theorem~\ref{t:Gromov} condition (3) is satisfied. Finally, let
  us turn to (4). By Theorem~\ref{t:Gromov} and the assumption (${\rm
    PCPS1}^*$) we know that every connected component of the Gromov
  boundary of $X$ contains the Gromov boundary of an
  apartment. Therefore, it suffices to show (4) for the Gromov
  boundary of an apartment. We observe that we find in every apartment
  a bi-infinite geodesic. This can be seen as follows: Let
  $(f_{-n},\ldots, f_{n})$ be a geodesic in an apartment $\Sigma \in
  {\mathcal A}$. By \cite[Theorem~1]{BP2} the face $f_{n}$ is not in
  $\mathrm{Cut}_{\Sigma}(f_{-n})$ and, therefore, there is $f_{n+1}
  \in \Sigma$ such that $(f_{-n},\ldots, f_{n+1})$ is a
  geodesic. Simultaneously, $f_{-n}$ is not in
  $\mathrm{Cut}_{\Sigma}(f_{n+1})$ and therefore there is $f_{-(n+1)}
  \in \Sigma$ such that $(f_{-(n+1)},\ldots, f_{n+1})$ is a geodesic
  in $\Sigma$. In this way , we construct a bi-infinite geodesic
  $(f_{n})_{n\in\Z}$. Let $\xi, \eta \in X(\infty)$ be the end points
  of the geodesics $(f_{n})_{n\ge0} \subset F_\Sigma$.  Since
  $(f_{n})_{n\in \Z}$ is a bi-infinite geodesic, we have
  $d(f_{n},f_{-n})))=2n$. So, we obtain for any $\eps>0$
  \begin{align*}
    d_{f_{0},\eps}(\xi,\eta) = \liminf_{n\to\infty}
    \exp(-\eps(n-\tfrac{1}{2}d(f_{n},f_{-n}))) =1.
  \end{align*}
  Hence, (4) is satisfied and we finished the proof.
\end{proof}

\begin{rmk}
  The question whether a Gromov hyperbolic space has positive Cheeger
  constant is very subtle. Note that every infinite tree $T$ is Gromov
  hyperbolic. But if we attach to one of its vertices the ray
  $[0,\infty)$ with integer vertices then the new tree $\widetilde
  T_1$ is still Gromov hyperbolic but it has vanishing Cheeger
  constant. This new ray adds an isolated point to the Gromov boundary
  of $T$ and therefore assumption (4) is violated for $\widetilde
  T_1$. On the other hand, if we attach to a sequence of vertices
  $(v_n)_{n \in \N}$ in $T$ the segments $[0,n]$ with integer vertices
  and denote the new tree by $\widetilde T_2$, then this new tree has
  again vanishing Cheeger constant. In this case both trees $T$ and
  $\widetilde T_2$ even have the same boundaries, but $\widetilde T_2$
  cannot have a quasi-pole since the newly added vertices do not lie
  in geodesic rays and, therefore, assumption (2) is violated (see end
  of Subsection 1.1 in \cite{C}).
\end{rmk}

The next result provides explicit lower bounds for the Cheeger
constant in terms of the face degrees and minimal and maximal thickness.

\begin{thm} \label{t:cheeger2} Let $X$ be a {\em locally finite}
  {\PCPS}. Then,
  \begin{align*}
    \al\ge \inf_{f \in F}
    \Big( \frac{m_{E}(f)}{M_{E}(f)} \Big(1-\frac{6}{{|\partial f|}}\Big) \Big)
    \ge \frac{m_{E}}{M_{E}}
    \Big( 1-\frac{6}{\min_{f \in F}{|\partial f|}} \Big).
    \end{align*}
    In particular, $\al>0$ if $|\partial f| \ge 7$ and $M_{E}<\infty$.
    Secondly,
  \begin{align*}
    \al \ge \inf_{f\in F} \frac{m_{E}(f)-2}{|f|}
    \ge \frac{m_{E}-2}{M_{F}}.
    \end{align*}
    In particular, $\al>0$ if $m_{E}\ge 3$ and $M_{F}<\infty$.
\end{thm}

The theorem implies in particular that all locally finite $2$-dimensional Euclidean buildings with minimal thickness $m_E \ge 3$ (i.e., every edge is contained in at least $4$ chambers) have positive Cheeger constant.
Moreover,  all locally finite hyperbolic buildings with generating polygon $P$ at least a $7$-gon have also
positive Cheeger constant.

\begin{proof} Translating \cite[Lemma~ 1.15]{DKa} into the 'dual'
  language (as the comment at the end of Section~\ref{s:PCPS}
  indicates) tells us that if there is a center $o \in V$ and $C \ge 0$
  such that
  \begin{align*}
    {|f|}_{+}-{|f|}_{-}\ge C|f|
  \end{align*}
  for all $f\in F$, then $\al\ge C$. Thus, it suffices to estimate
  $\inf_{f\in F}({|f|}_{+}-{|f|}_{-})/|f|$ to get a lower bound on
  $\al$. For $f\in F$, let $n\ge0$ be such that $f\in S_{n}$ and let
  $\Sigma$ be an apartment that contains $f$. By
  Proposition~\ref{prop:bigon} we immediately have ${|f|}_{-}\leq
  2$. Moreover, by \cite[Theorem~3.2]{BP1} (combined with
  Theorem~\ref{t:Cut_locus}) there are at most two neighbors of $f$ in
  $F_{\Sigma}\cap S_{n}$ and, therefore, $ {|f|}_{+}\geq
  m_{E}(f){|f|}_{\Sigma,+} \ge m_{E}(f)( {|\partial f|}-4 )$. Here
  $|f|_{\Sigma,+}$ denotes the number of forward neighbors of $f$
  within $\Sigma$, which is $|\partial f|$ minus the number ($\le 2$)
  of backward neighbors of $f$ minus the number ($\le 2$) of
  neighbors of $f$ in $F_{\Sigma}\cap S_{n}$. Moreover, $|f|\leq
  M_{E}(f){|\partial f|}$.  Hence,
  \begin{align*}
    \frac{{|f|}_{+}-{|f|}_{-}}{|f|} \ge
    \frac{m_{E}(f)}{M_{E}(f)} \left(1-\frac{6}{|\partial f|}\right)
  \end{align*}
  which yields the first inequality. On the other hand, we have by
  Theorem~\ref{t:Cut_locus} and Lemma~\ref{l:neighbors}~(a)
  ${|f|}_{+}\ge m_{E}(f)$.  Hence, by ${|f|}_{-}\leq 2$
  \begin{align*}
    \frac{{|f|}_{+}-{|f|}_{-}}{|f|}\ge \frac{ m_{E}(f)-2}{|f|}
  \end{align*}
  This finishes the proof.
\end{proof}

From the proof we may easily extract the following statement which
turns out to be useful for studying the essential spectrum of the
Laplacian. Define for a locally finite polygonal complex $X=(V,E,F)$
the {\em Cheeger constant at infinity} by
\begin{align*}
  \al_{\infty} = \sup_{K\subseteq F\,\mathrm{finite}} \al_{F\setminus K}.
\end{align*}

\begin{cor} \label{c:cheeger}
  Let $X$ be a locally finite {\PCPS}. Then
  \begin{align*}
    \al_{\infty}\ge \sup_{K \subseteq F\,\mathrm{finite}} \inf_{f \in F \backslash K}
    \frac{m_{E}(f)}{M_{E}(f)} \Big(1-\frac{6}{|\partial f|}\Big)
  \end{align*}
\end{cor}


\subsection{Finiteness and infiniteness}

In this subsection we show that positivity or non-positivity of {\em
  sectional face curvature} determines whether a locally finite
polygonal complex with planar/spherical substructures
is finite or infinite. The statement that positive curvature
implies finiteness is an analogue of a theorem of Myers for Riemannian
manifolds \cite{M}.

\begin{thm} \label{t:myers}
  Let $X=(V,E,F)$ be a locally finite {\PCPSSS} with
  apartment system ${\mathcal A}$.
  \begin{itemize}
  \item[(a)] If we have $\kappa^{(\Sigma)}(f) > 0$ for all $\Sigma \in
    {\mathcal A}$ and all $f \in F_\Sigma$, then $F$ is finite and $X$
    is a {\PCSS}.
  \item[(b)] If we have $\kappa^{(\Sigma)}(f) \le 0$ for all $\Sigma
    \in {\mathcal A}$ and all $f \in F_\Sigma$, then $F$ is infinite
    and $X$ is a {\PCPS}.
  \end{itemize}
\end{thm}

\begin{proof} Note first that every planar tessellation has infinitely
  many faces (since the closure of every face is compact) while every
  spherical tessellation has finitely many faces. Therefore,
  $F_\Sigma$, $\Sigma \in {\mathcal A}$, is infinite if $X$ is a
  {\PCPS} and finite if $X$ is a {\PCSS}.

  We first prove (b) by contraposition. Assume that $X$ is a {\PCPSSS} with $F$ a finite set. We
  will show that there is a face with positive sectional face   curvature. Choose an apartment $\Sigma \in {\mathcal A}$. By the
  Gau\ss-Bonnet theorem, we have
  \begin{align*}
    \sum_{f \in F_\Sigma} \ka^{(\Sigma)}(f) = \chi(\Sp^2) = 2,
  \end{align*}
  where $\chi$ denotes the Euler characteristic. Hence, $\ka^{(\Sigma)}$
  must be positive on some faces. This shows (b).

Turning to (a), we assume that $\ka^{(\Sigma)}(f)>0$ for all $\Sigma \in
  {\mathcal A}$ and all $f \in F_\Sigma$. By DeVos-Mohar's proof of
  Higuchi's conjecture \cite[Theorem 1.7]{DM} (which is again stated
  in the dual formulation) every apartment must be finite. Moreover,
  the number of faces (in their case vertices) in an apartment is
  uniformly bounded by 3444 except
  for prisms and antiprisms\footnote{Note that in the meantime the
    bound has been improved by Zhang \cite{Z} to 580 vertices while
    the largest known graphs with positive curvature has 208 vertices
    and was constructed by Nicholson and Sneddon \cite{NS}.}. A prism in our dual setting are two
  wheels of triangles glued together along their boundaries and an
  antiprism are two wheels of squares glued together along their
  boundaries (see Figure \ref{wheels}). We can think of these two
  wheels as representing the lower and upper hemisphere of ${\mathbb
    S}^2$ and the boundaries as agreeing with the equator of the
  sphere ${\mathbb S}^2$.

  \begin{figure}[h]
    \begin{center}
      \scalebox{0.4}{\includegraphics{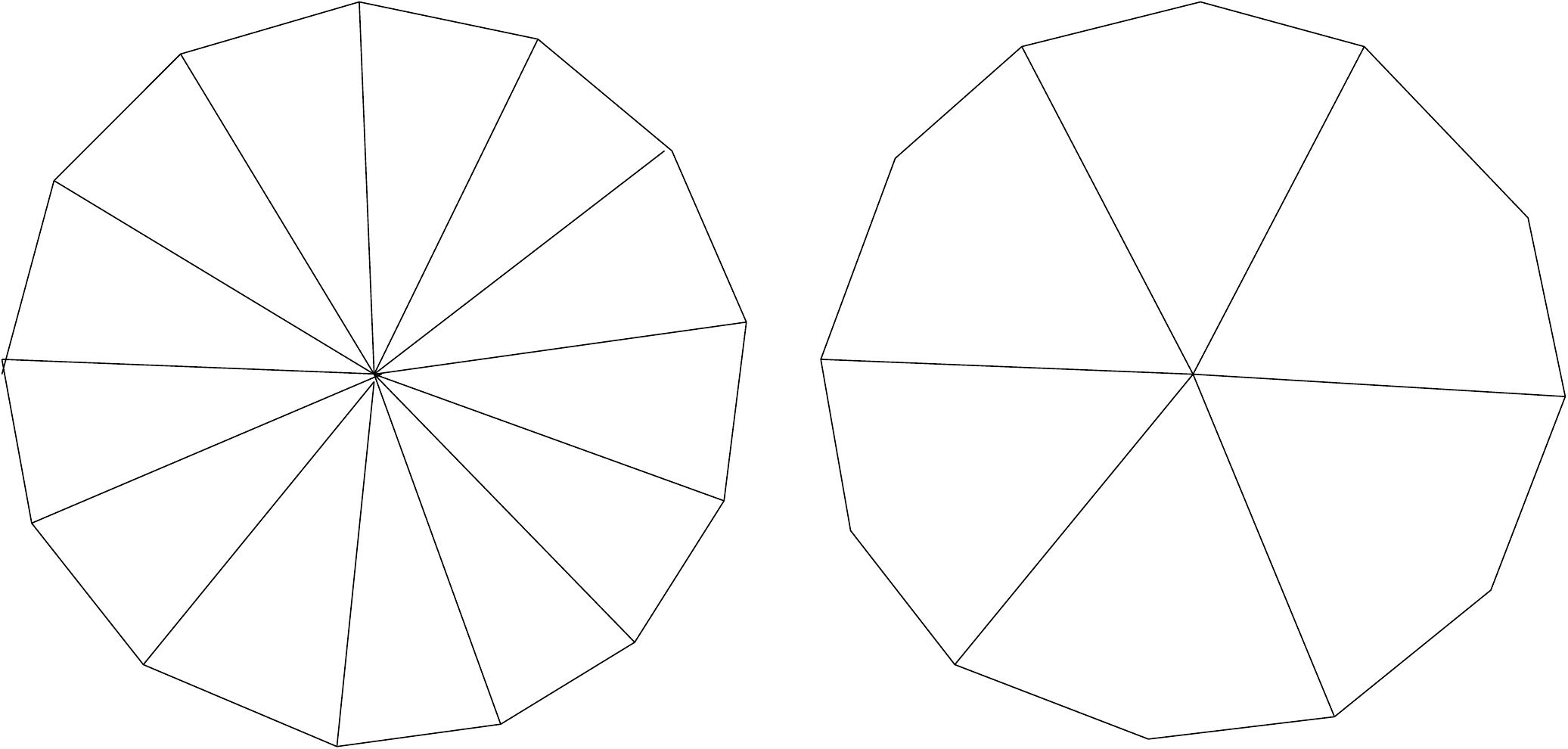}}
    \end{center}
    \caption{A wheel of triangles and a wheel of squares}
    \label{wheels}
  \end{figure}

  If $F$ is infinite, then there exists a face $f_0 \in F$ and a
  sequence of faces $f_n \in F$ with $d(f_0,f_n) \to \infty$ because
  of the local finiteness. Then $f_0$ must lie in a sequence
  $\Sigma_k$ of spherical apartments ${\mathbb S}^2$ tessellated by
  pairs of wheels with number of faces going to infinity, glued
  together along the equator. Assuming that $f_0$ lies always in the
  lower hemisphere of $\Sigma_k \cong {\mathbb S}^2$, then the south
  pole of all these apartments would be one and the same vertex $v_0
  \in \overline{f}_0$. But this would imply that $|v_0| = \infty$,
  which contradicts the local finiteness. Therefore, $F$ must be
  finite which implies that $X$ is a {\PCSS}.
\end{proof}

\section{Spectral theory} \label{s:spectrum}

In this section we turn to the spectral theory of the Laplacian on
polygonal complexes. As the geometric structure is determined by
assumptions on the faces it is only natural to consider the Laplacian
for functions on the faces. The reader who prefers to think about the
Laplacian as an operator on functions on the vertices is referred to
comment at the end of Section~\ref{s:PCPS}. That is we can associate a
graph $G_X$ to each polygonal complex $X=(V,E,F)$ in a
natural way.

Let $X=(V, E,F)$ be a locally finite polygonal complex and
\begin{align*}
  \ell^{2}(F)=\{\ph:F\to\C\mid \sum_{f\in F}|\ph(f)|^{2}<\infty\}.
\end{align*}
For functions $\ph,\psi\in\ell^{2}(F)$ the standard scalar product is
given by
\begin{align*}
  \langle \ph,\psi\rangle=\sum_{f\in F}\overline{\ph(f)}\psi(f),
\end{align*}
and the norm is given by $\|\ph\|=\sqrt{ \langle \ph,\ph\rangle}$.
Define the Laplacian $\Delta$ by
\begin{align*}
  \Delta\ph(f)=\sum_{g\in F, g\sim f}(\ph(f)-\ph(g))
\end{align*}
for functions in the domain
\begin{align*}
  D(\Delta)=\{\psi\in \ell^{2}(F)\mid \Delta\psi\in \ell^{2}(F)\}.
\end{align*}
It can be checked directly that the operator is positive and,
moreover, it is selfadjoint by \cite[Theorem~1.3.1]{Woj}.
Note that the operator $\Delta$ can be seen to coincide with the graph Laplacian on $\ell^{2}(G_{X})$.

By standard Cheeger estimates \cite{Kel} based on \cite{DKe,Fu} we have
\begin{align*}
  \lm_{0}(\Delta)\ge m_{F}(1-\sqrt{1-\al^{2}}),
\end{align*}
where $\lm_{0}(\Delta)$ denotes the bottom of the spectrum of $\Delta$
and
$$ m_{F} = \min_{f\in F} |f|.$$
Applying Theorem~\ref{t:cheeger1} gives a criterion when the bottom of
the spectrum is positive and Theorem~\ref{t:cheeger2} even gives
explicit estimates.


\subsection{Discreteness of spectrum and eigenvalue asymptotics}

\label{S:specdisc}

In this subsection we address the question under which circumstances the
spectrum of $\Delta$ is purely discrete. We prove an analogue of a
theorem of Donnelly-Li, \cite{DL}, for Riemannian manifolds that
curvature tending to $-\infty$ outside increasing compacta implies
emptiness of the essential spectrum.

For a selfadjoint operator $T$ we denote the eigenvalues below the
essential spectrum in increasing order counted with multiplicity by
$\lm_{n}(T)$, $n\ge0$. For two sequences $(a_{n})$, $(b_{n})$ we write
$a_{n}\sim b_{n}$ if there is $c>0$ such that $c^{-1}a_{n}\leq
b_{n}\leq c a_{n}$. We denote the maximal operator of multiplication
by the face degree by $D_{F}$. That is $D_{F}$ is an operator from
$\{\ph\in\ell^{2}(F)\mid |\cdot|\ph\in\ell^{2}(F)\}$ to $\ell^{2}(F)$
acting as
\begin{align*}
  D_F \ph(f)=|f|\ph(f).
\end{align*}
We call $X$ \emph{balanced} if there is $C > 0$ such that $C
m_{E}(f)\ge M_{E}(f)$ and \emph{strongly balanced} if
\begin{align*}
  \sup_{K \subseteq F\, \mathrm{finite}}
  \inf_{f \in F \setminus K}\frac{m_{E}(f)}{M_{E}(f)}=1.
\end{align*}
That means that $C$ in the definition of balanced equals $1$ asymptotically.
An analogue of the Donnelly-Li result reads as follows. Let
\begin{align*}
  \ka_{\infty} := \inf_{K \subseteq F \, \mathrm{finite}} \sup_{\Sigma
    \in {\mathcal A}, f \in F_{\Sigma}\setminus K} \ka^{(\Sigma)}(f).
\end{align*}

  \begin{thm} \label{t:DiscreteSpectrum} Let $X=(V,E,F)$ be a {\em
      locally finite} {\PCPS} that is balanced. If $\ka_{\infty} = -
    \infty$ then the spectrum of $\Delta$ is purely discrete and
  $$\lm_{n}(\Delta) \sim \lm_{n}(D_{F}).$$
  If, additionally, $X$ is
  strongly balanced, then $$\frac{\lm_{n}(\Delta)}{ \lm_{n}(D_{F})}\to 1\mbox{ as
  }n \to \infty.$$
  Finally, under the additional assumption
  $M_{E}<\infty$, purely discrete spectrum of $\Delta$ implies
  $\ka_{\infty} = -\infty$.
\end{thm}

We like to mention that the result here holds for the {\em generally
  unbounded} discrete Laplacian. The first result on the essential
spectrum of graphs analogous to Donnelly-Li was proved by Fujiwara
\cite{Fu} and he considered the {\em normalized}
Laplacian. The very different spectral behavior of these two operators
is discussed in \cite{Kel}.

The proof of Theorem \ref{t:DiscreteSpectrum} is based on the
following proposition.

\begin{prop} \label{p:Delta} Let $X=(V,E,F)$ be a locally finite
  {\PCPS}. If
  \begin{align*}
    a := \sup_{K\subseteq F\,\mathrm{finite}} \inf_{f\in F\setminus K}
    \frac{m_{E}(f)}{M_{E}(f)} \Big(1-\frac{6}{|\partial f|}\Big) > 0,
  \end{align*}
  then the spectrum of $\Delta$ is discrete if and only if
  \begin{equation*} \label{eq:facedegreeinf}
    \sup_{K\subseteq F\,\mathrm{finite}} \inf_{f\in F\setminus K}|f| = \infty.
  \end{equation*}
  In this case,
  \begin{align*}
  (1-\sqrt{1-a^2}) \leq
  \liminf_{n\to\infty} \frac{\lm_{n}(\Delta)}{\lm_{n}(D_{F})} \leq
  \limsup_{n\to\infty} \frac{\lm_{n}(\Delta)}{\lm_{n}(D_{F})} \leq
  (1+\sqrt{1-a^2}).
  \end{align*}
\end{prop}

\begin{proof} The characterization of discreteness of spectrum follows
  from Corollary~\ref{c:cheeger} and \cite[Theorem~2]{Kel}. The
  asymptotics of eigenvalues follow combining
  Corollary~\ref{c:cheeger} and \cite[Thms.~2.2. and 5.3.]{BGK}.
\end{proof}

\begin{proof}[Proof of Theorem~\ref{t:DiscreteSpectrum}] We observe
  that for all $\Sigma\in \A$ and $f\in F_{\Sigma}$
  \begin{align*}
    -\frac{|f|_\Sigma}{2} \leq \ka^{(\Sigma)}(f).
  \end{align*}
  Hence, $\ka_{\infty}=-\infty$ implies $\sup_{K\subseteq
    F\,\mathrm{finite}} \inf_{f\in F\setminus K}|f|=\infty$.
  Combining this with the assumption that $X$ is balanced with
  constant $C$ implies that $a\ge 1/C$, where $a$ is taken from
  Proposition~\ref{p:Delta}. In the case of $X$ being strongly
  balanced we have $a=1$. Thus, the first part of the theorem follows
  from Proposition~\ref{p:Delta}. Conversely, if there is $c>0$ such that
  $\ka_{\infty} \ge -c > -\infty$ then there is a sequence of faces
  $f_{n}$ with $d(f,f_n) \to \infty$ for any fixed face $f \in F$ and
  apartments $\Sigma_{n}$, $n \ge 0$, such that
  \begin{align*}
    -c < \ka^{(\Sigma_{n})}(f_{n}) \leq 1 - \frac{|f|_\Sigma}{6}
    \leq 1 - \frac{|f|}{6M_{E}},
  \end{align*}
  where we used ${|v|}_{\Sigma}\ge3$ which holds as $\Sigma$ is a tessellation. We conclude
  that $|f_{n}|$ is uniformly bounded by some constant $c'>0$.  Thus,
  the essential spectrum of $\Delta$ starts below $c'$ (confer
  \cite[Theorem~1]{Kel}) and $\Delta$ does not have purely discrete
  spectrum.
\end{proof}

\begin{example}
  The simplest example of a {\PCPS} satisfying the conditions of Theorem
  \ref{t:DiscreteSpectrum} is a planar tessellation $X=(V,E,F)$ with
  one apartment $\Sigma=X$ and center $o \in F$ such that $\lim_{n \to
    \infty} \inf_{f \in S_n} |\partial f| = \infty$. In this case we
  have
  $$ \kappa^{(\Sigma)}(f) \le 1 - \frac{|\partial f|}{6}, $$
  and we see that $\kappa_\infty = - \infty$. Moreover, $X$ is
  strongly balanced since we have $m_E(f) = M_E(f) = 1$. Therefore,
  the spectrum of $\Delta$ is purely discrete and
  $\lambda_n(\Delta)/\lambda_n(D_F) \to 1$.

  Note that purely discrete spectrum can also be established by
  increasing $m_E(f)$ instead of $|\partial f|$ for all faces outside
  compact sets (by keeping the polygonal complex balanced) and
  applying Proposition \ref{p:Delta} directly. The condition
    $\sup_{K\subseteq F\,\mathrm{finite}} \inf_{f\in F\setminus K}|f| = \infty$
follows then directly from $|f| \ge
  m_E(f)$.
\end{example}


\subsection{Unique continuation of eigenfunctions}

While unique continuation results hold in great generality for
continuum models with very mild assumptions, there are very natural
examples for graphs with finitely supported eigenfunctions, see
\cite{DLMSY} and various other references. In this subsection we prove
that for non-positive curvature there are no finitely supported
eigenfunctions.

\begin{thm} \label{t:uniqueContinuation}
  Let $X=(V,E,F)$ be a {\em locally finite} {\PCPS} such that
  $\ka_{c}^{(\Sigma)} \leq 0$ for all $\Sigma \in {\mathcal A}$. Then,
  $\Delta$ does not admit finitely supported eigenfunctions.
\end{thm}

Cases where we do not have finite supported eigenfunctions are
therefore Example \ref{ex:vdov} and Examples
\ref{ex:prodtrees}-\ref{ex:hag}.

In \cite{KLPS,Kel2} results like Theorem~\ref{t:uniqueContinuation} are
found for the planar case and more general operators. Indeed, we
consider here also nearest neighbor operators, where we even do not have to assume local finiteness.

\begin{dfn}
  Let $X=(V,E,F)$ be a polygonal complex. We call $A$ a \emph{nearest
    neighbor operator} on $X$ if there is $a:F\times F\to \C$
  \begin{itemize}
    \item [(NNO1)] $a(f,g)\neq 0$ if $f\sim g$.
    \item [(NNO2)] $a(f,g)=0$ if $f\not\sim g$.
    \item [(NNO3)] $\sum_{g\in F} |a(f,g)|<\infty$ for all $f\in F$.
  \end{itemize}
  and $A$ acts as
  \begin{align*}
    A\ph(f)=\sum_{g\in F}a(f,g)\ph(g),
  \end{align*}
  on functions $\ph$ in
  \begin{align*}
    \ow D(A)=\{\ph: F \to \C \mid
    \sum_{g\in F} |a(f,g)\ph(g) | < \infty \quad \forall \, f \in F\}.
  \end{align*}
\end{dfn}

The summability assumption (NNO3) guarantees that the functions of
finite support are included in $\ow D(A)$. Clearly, the Laplacian
introduced at the beginning of this section is a nearest neighbor operator,
where we can also add an arbitrary potential to be in the general
setting of Schr\"odinger operators. Theorem~\ref{t:uniqueContinuation}
is an immediate consequence of the following theorem.

\begin{thm} \label{t:uniqueContinuation2}
  Let $X=(V,E,F)$ be a {\PCPS} such that $\ka_{c}^{(\Sigma)} \leq 0$ for
  all $\Sigma \in {\mathcal A}$ and $A$ be a nearest neighbor operator on
  $X$. Then $A$ does not admit eigenfunctions supported within a   distance ball.
\end{thm}

\begin{proof}
  Let $\ph\in \ow D(A)$ be an eigenfunction of $A$ to the eigenvalue
  $\lm$.  Let $k$ be such that $\ph$ vanishes completely on all
  distance spheres at levels larger or equal than $k$ from a center $o
  \in F$. Let $f_{0} \in F$ be a face at distance $k-1$. We want to
  show that $\ph(f_{0}) = 0$. Let $\Sigma$ be an apartment containing
  $o$ and $f_{0}$. Since we do not have cut-locus in any of the
  apartments due to non-positive sectional corner curvature, cf.\@
  Theorem~\ref{t:Cut_locus}, there exists a face $g_{0}\in F_{\Sigma}$
  adjacent to $f_{0}$ with $d(o,g_{0}) = k$.  By assumption, we have
  $\ph(g_{0}) = 0$. Now, by convexity, all faces $f \in F$ with
  $d(f,o) = k-1$ adjacent to $g_{0}$ lie within $\Sigma$.  By
  Proposition~\ref{prop:bigon} there can be at most two such faces,
  one of them equal to $f_{0}$. If there is only one such face, namely
  $f_{0}$, we conclude from the eigenfunction identity evaluated at
  $g_{0}$ that we have $\ph(f_{0}) = 0$.  If there are two such faces,
  say $f_{0},f_{1}$, then we conclude from the eigenfunction identity
  evaluated at $g_{0}$ that $a(g_{0},f_{0})\ph(f_{0}) = -
  a(g_{0},f_{1}) \ph(f_{1})$.  With the notation of
  \cite[Section~2.2]{BP2} the vertex $v_{0}$ in the intersection of
  $\overline{f_{0}}, \overline{f_{1}}$ and $\overline{g_{0}}$ has
  label $b$ with respect to the tessellation $\Sigma$ (label $b$ means
  that there is more than one adjacent face to $v_{0}$ within
  $B_{k-1}$ or if one of the faces adjacent to $v_{0}$ is a triangle
  then there are even more than three adjacent faces in $B_{k-1}$;
  however, the case that $v_{0}$ has a neighboring triangle can be
  excluded by $\ka_{c}^{(\Sigma)} \leq 0$). The vertex $v_{0}$ has two
  neighbors in the boundary of $B_{k-1}$ in $\Sigma$.  One of these
  neighbors is in the intersection of $\overline{f_{0}}\cap
  \overline{g_{0}}$ and the other one which we denote by $v_{1}$ is in
  the intersection of $\overline{f_{1}}\cap \overline{g_{0}}$. By
  \cite[Cor.~2.7.]{BP2} the vertex $v_{1}$ has label $a^{+}$ (which
  means that $v_{1}$ has only one adjacent face within $B_{k-1}$).
  This implies that the face $f_{1}$ has another neighbor $g_{1}$ in
  $S_{k}$. By assumption $\ph(g_{1})=0$ and applying the same
  arguments to $g_{1}$ we find $f_{2}\in S_{k-1}\cap F_{\Sigma}$,
  $f_{2}\sim g_{1}$ such that $a(g_{1},f_{1})\ph(f_{1}) = -
  a(g_{1},f_{2}) \ph(f_{2})$. Proceeding inductively we find the
  sequences $(f_{0},\ldots,f_{n})$, $f_{0}=f_{n}$, and
  $(g_{0},\ldots,g_{n})$, $g_{0}=g_{n}$ of faces in $\Sigma$ that form
  a closed boundary walk and boundary vertices
  $(v_{0},\ldots,v_{2n})$, $v_{0}=v_{2n}$, with labels
  $b,a^{+},b,a^{+},b,\ldots$. However, this is geometrically
  impossible \cite[Prop.~13]{KLPS}.  Hence, we conclude
  $\ph(f_{0})=0$. As this argument applies for all faces in $S_{k-1}$
  we deduce that $\ph$ vanishes on $S_{k-1}$. Repeating this argument
  for $S_{k-j}$, $j=2,\ldots,k$, yields that $\ph$ vanishes on $B_{k}$
  and thus by assumption on $F$. We finished the proof.
\end{proof}

We conclude this subsection by giving examples of tessellations with
negative sectional face curvature that admit finitely supported
eigenfunctions. This shows the assumption in the theorem cannot be
modified to negative sectional face curvature instead of non-positive
sectional corner curvature.

\begin{example} Let $\Sigma_{n}$, $n\ge3$, be a bipartite tessellation
  of the plane $\R^2$ with squares as follows. There are
  two infinite sets of vertices $V_{1}$ and $V_{2}$, where the
  vertices in $V_{1}$ have degree $2n$ and the vertices in $V_{2}$
  have degree $3$. The tessellation $\Sigma_{n}$ is now given such
  that vertices in $V_{1}$ are only connected to vertices in $V_{2}$
  and vice versa.  Hence, each face contains two vertices of $V_{1}$
  and two of $V_{2}$. See Figure \ref{Sigma4} for the tessellation
  $\Sigma_4$, realized in the hyperbolic Poincar{\'e} unit disk.

  \begin{figure}[h]
    \begin{center}
      \psfrag{0}{$0$}
      \psfrag{1}{$1$}
      \psfrag{-1}{$-1$}
\scalebox{.25}{\includegraphics{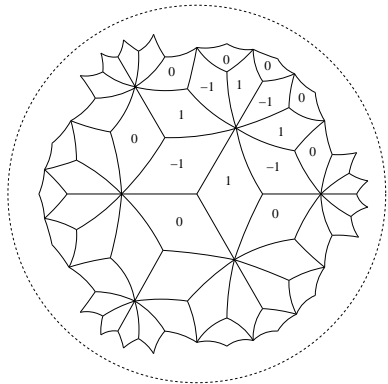}}
    \end{center}
    \caption{Part of the tessellation $\Sigma_{4}$ with a finitely supported eigenfunction which is zero at the faces with no entry.}
    \label{Sigma4}
  \end{figure}

  The face curvature is then given by
  \begin{align*}
    \ka(f) = 1 -\frac{|f|}{2} +
    \sum_{v \in f} \frac{1}{|v|} = 1 - 2 + \frac{2}{3}+\frac{2}{2n}=
    -\frac{n-3}{3n}.
  \end{align*}
  For $n>3$ the face curvature is negative and in the interval
  $(-1/3,-1/12)$. On the other hand, we have for the corner curvatures
  $$ \ka_c(v_1,f) = - \frac{n-2}{4n}, \quad \ka_c(v_2,f) = \frac{1}{12} > 0, $$
  with $v_1 \in V_1$ and $v_2 \in V_2$ and $v_1,v_2 \in
  \overline{f}$. Moreover, for a vertex with degree $2n$ let
  $F_{0}=\{f_{1},\ldots,f_{2n}\}$ be the faces around it in cyclic
  order. Let a function $\ph$ with support in $F_{0}$ be given such
  that $\ph(f_{2j})=1$ and $\ph(f_{2j-1})=-1$ for $j=1,\ldots,
  n$. Then, $\ph$ is a finitely supported eigenfunction of $\Delta$ to
  the eigenvalue $6$. Looking at the dual regular graph ${\Sigma_n}^*$
  with constant vertex degree $4$, we see that the
  $\Delta$-eigenfunction $\varphi$ of $\Sigma_n$ corresponds to an
  eigenvector of the adjacency matrix of ${\Sigma_n}^*$ to the
  eigenvalue $-2$.
\end{example}


\subsection{The Dirichlet problem at infinity}

We assume that $X=(V,E,F)$ is a {\PCPS} with strictly negative sectional
corner curvature and that (${\rm PCPS1}^*$) holds. Moreover, we assume
$M_V, M_F < \infty$. Then we know from Theorem \ref{t:Gromov} that
$(X,d)$ is Gromov hyperbolic and that the boundary $X(\infty)$ carries
a natural topological structure. Moreover, $\overline{X} = X \cup
X(\infty)$ is compact (see \cite[Prop. III.H.3.7(4)]{BH}). Given a
function $U\in C(X(\infty))$, the {\em Dirichlet problem at infinity}
asks whether there is a unique continuous function $u \in
C(\overline{X})$ which agrees with $U$ on $X(\infty)$ and such that
the restriction $u_0 = u\vert_{X}$ is harmonic (i.e., $\Delta u =
0$). The existence of such a function $u$ is the main problem since uniqueness of the solution follows from the maximum principle. Applying the general theory of \cite{Anc} to Theorem~\ref{t:cheeger1} answers this question positively.

\begin{thm} \label{thm:Dirichlet} Let $X=(V,E,F)$ be a {\PCPS} such that
  $\kappa_c^{(\Sigma)} < 0$ for all $\Sigma \in {\mathcal A}$. Assume
  that $X$ additionally satisfies (${\it PCPS1}^*$) and $M_{V},
  M_{F}<\infty$. Then $(X,d)$ is Gromov hyperbolic and the Dirichlet problem at infinity is solvable on $X$. In particular,   there are infinitely many linearly independent bounded non-constant   harmonic functions on $X$.
\end{thm}

For spaces, where the theorem is applicable and the Dirichlet problem at
infinity can be solved, are all locally finite $2$-dimensional hyperbolic buildings
with regular hyperbolic polygons as faces.

\begin{proof}
  Gromov hyperbolicity of $(X,d)$ follows from Theorem \ref{t:Gromov}.
  Let $P$ denote the averaging operator $P\varphi(f) = \frac{1}{|f|}
  \sum_{g \sim f}\varphi(g)$. It is easy to see that $P$ satisfies the
  properties of \cite[Assumptions 1.1]{Anc}. Note further that a
  function $\varphi$ on $F$ satisfies $\Delta \varphi = 0$ if and only
  if $P \varphi = \varphi$. We know from Theorem \ref{t:cheeger1} that
  the Cheeger constant $\alpha$ of $X$ is positive. We conclude from
  \cite[Prop. 4.4]{Anc} that the crucial condition (*) in \cite{Anc}
  is therefore satisfied. Moreover, we deduce from Gromov
  hyperbolicity of $(X,d)$ and \cite[Cor. 6.10]{Anc} that the
  assumptions (G.A) in \cite[Theorem~5.2]{Anc} are satisfied and the
  Gromov compactification agrees with the $P$-Martin
    compactification of $X$.  Then the statement follows from
  \cite[Cor. 5.4]{Anc}.
\end{proof}

\begin{rmk}
  We have already mentioned in the proof of Theorem \ref{thm:Dirichlet}
  that Ancona's theory also implies that the Gromov and the geodesic
  boundary of $(X,d)$ agrees with the {\em $P$-Martin boundary}
  boundary. The $P$-Martin boundary is an analytically defined
  boundary based on asymptotic properties of Green's functions $G: F
  \times F \to [0,\infty)$ (see \cite[Section V]{Anc}).
\end{rmk}

\section{Examples}\label{s:examples}

In this section, we will mainly focus on non-positively curved
{\PCPSs}. Rich classes of examples are provided by $2$-dimensional
Euclidean and hyperbolic buildings. Before we consider these classes
more closely, let us start with particularly simple examples of
non-buildings.

\subsection{Simple examples and basic notions}

As mentioned earlier, every planar tessellation $\Sigma = (V,E,F)$ is
trivially a {\PCPS} with just one apartment, i.e., ${\mathcal A} = \{
\Sigma \}$.

Next, let us introduce {\em morphisms} between two complexes $X_1$ and
$X_2$: These are continuous maps from $X_1$ to $X_2$ mapping $k$-cells
of $X_1$ homeomorphically to $k$-cells of $X_2$, for all $k$. A morphism $f: X_1 \to X_2$ is an {\em isomorphism} if both $f$ and
$f^{-1}$ are morphisms. In this case we call $X_1$ and $X_2$
isomorphic complexes.

\begin{example}[``Book'']
Let ${\mathcal H} = (V,E,F)$ be the tessellation of the upper half
space $\{ (x,y) \in \R^2 \mid y \ge 0 \}$ where
$$ V = \{ (x,y) \in \Z^2 \mid y \ge 0 \}, $$
$E$ is the set of horizontal and vertical straight Euclidean line
segments of length 1 connecting two vertices of $V$, and $F$ is the
set of all Euclidean unit squares with vertices in $V$. Let $k\ge 2$ be an
integer and $X_k$ be the polygonal complex obtained by taking
$k$ copies of ${\mathcal H}$ and identifying them along their
boundaries $\R \times \{0\} \subset {\mathcal H}$. We can think of
$X_k$ as a book with the copies of ${\mathcal H}$ as its pages. Note
that the union of any two pages can be understood as a tessellation of
the plane by squares. Every such choice represents an apartment of the
{\PCPS} $X_k$. It is straightforward to see that $X_k$ has non-positive
sectional corner curvature. Books can also be obtained by combining
pages with more general and different polygonal structures by using
isomorphisms between their boundaries (considered as $1$-dimensional
cell complexes). Moreover, it is also possible to consider books with
{\em infinitely many} pages. They are obviously non-locally finite
{\PCPSs}.
\end{example}

\begin{example} Let us present an example of polygonal complexes that
  have no planar substructures satisfying (PCPS1) and (PCPS2). Let
  $X=(V,E,F)$ be given by $V=\Z^3$, $E$ be the set of straight
  Euclidean line segments of length 1 connecting two vertices of $V$,
  and $F$ be the set of all unit squares with vertices in $V$. The
  triple $X$ is obviously a polygonal complex, but there does not
  exist a choice of apartments (planes tessellated by squares)
  satisfying both conditions (PCPS1) and (PCPS2). The set of all
  planes parallel to the coordinate planes does not satisfy
  (PCPS1). Thus, we also need to declare certain topological planes
  which are bent to be apartments. But it is easy to see that the
  convexity property (PCPS2) is violated for any such bent plane.
\end{example}

Next, we come to two important notions in the local combinatorial description of polygonal complexes. Our purpose is to use these notions later to define certain buildings in the next sections.

\begin{dfn}[Link] \label{dfn:link}
  Let $X = (V,E,F)$ be a polygonal complex. The \emph{link} $L(v)$ of a vertex $v \in V$ is a graph defined as  follows: Every   edge adjacent to $v$ is represented by a vertex in $L(v)$, and two   vertices $w_1,w_2$ in $L(v)$ are connected by an edge in $L(v)$ if   the edges in $X$ corresponding to $w_1,w_2$ are edges of a face $f$   in $F$.
\end{dfn}

As an easy example one finds that the link of a vertex of degree $d$ in a planar tessellation is a $d$-gon.  Similarly, on finds that the link of a vertex in  $\Z^{3}$ is an octahedron.

Furthermore, polygonal complexes are often described via the type of their faces and the graphs appearing as links. It is proven in \cite[Theorem~ 1]{BB1} that for given
$p \ge 6$, $n \ge 3$ there is a continuum of non-isomorphic simply connected polygonal complexes such that faces are $p$-gons and the links  of all vertices are
the $1$-skeletons of an $n$-simplex.

Next, we give the definition of generalized $m$-gons that appear as links of Euclidean and hyperbolic buildings which are introduced in the next section.

\begin{dfn}[Generalized $m$-gon] Let $m \ge 2$ be an integer. A {\em
  generalized $m$-gon} is a connected bipartite graph of diameter $m$
and of girth $2m$ such that each vertex has degree $\ge 2$.
\end{dfn}

Next to ordinary $2m$-gons, important examples of generalized $m$-gons are the Heawood graph ($m=3$) and complete bipartite graphs ($m=2$). As it shall be discussed in the next sections, they appear as examples of links of vertices of buildings.

Let us make another remark to stress the relevance of these notions. The adjacency matrices of regular generalized $m$-gons have interesting spectral properties. In particular, they are Ramanujan
graphs (see \cite[Section 8.3]{Lub}). Spectral properties of the links  of vertices of $2$-dimensional simplicial complexes were also very useful to obtain Kazdhan property (T) for groups acting cocompactly in these complexes (see \cite{BaSw}).

\subsection{Euclidean and hyperbolic buildings}

Let us give a quick introduction into $2$-dimensional Euclidean and
hyperbolic buildings, following essentially \cite{GP}. In contrast to
our Definition \ref{dfn:polycomp}, the cells in the polygonal
complexes used for Coxeter complexes and buildings have an additional
metric structure, namely, the $1$-cells are open Euclidean or
hyperbolic geodesic segments and the $2$-cells are Euclidean or
hyperbolic polygons (we restrict our considerations to compact ones),
and the attaching maps are isometries (see also
\cite[Sct. I.7.37]{BH}). We call an {\em isometric isomorphism}
between two polygonal complexes an {\em isometry}, for simplicity. The
closures of the $2$-cells are called {\em chambers} of the polygonal
complex.

Important {\em planar} polygonal complexes are {Coxeter
  complexes}, which we introduce first (for more details see, e.g.,
\cite{Hum}). Let ${\mathbb X}$ stand for either the Euclidean plane
$\R^2$ or the hyperbolic plane ${\mathbb H}^2$. Let $P \subset
{\mathbb X}$ be a compact polygon with $k \ge 3$ vertices such that the
interior angle at vertex $i$ is of the form $\pi/m_i$ with $m_i \ge
2$. We call such a polygon $P$ a {\em Coxeter polygon}. Let $S = \{s_1,\dots,s_k\}$ be the
set of reflections along the sides of $P$ and $W$ be the group generated by the elements
of $S$. Then it is a well known fact due to Poincar{\'e} that $W$ is a discrete subgroup of the isometry group ${\rm Iso}(\mathbb X)$ with $P$ as its fundamental domain, i.e., the translates $\{ gP \mid g \in W \}$ form a tessellation of $\mathbb X$, which is a planar polygonal
complex in the above sense. We refer to it as the \emph{Coxeter complex} $C(W,S)$ and
call the polygon $P$ the {\em generating polygon} of the {\em Coxeter group} $(W,S)$.

\begin{dfn}[Building]
  Let ${\mathbb X} \in \{\R^2,{\mathbb H}^2\}$, $P \subset {\mathbb
    X}$ be a Coxeter polygon and $(W,S)$ be the associated Coxeter group. A ($2$-dimensional) {\em building of type
    $(W,S)$} is a polygonal complex $X=(V,E,F)$, together with a  set ${\mathcal A}$ of
  subcomplexes whose elements $\Sigma = (V_\Sigma,E_\Sigma,F_\Sigma)$
  are called {\em apartments}, with the following properties:
  \begin{itemize}
  \item [(B1)] For any two cells of $X$ there is an apartment containing     both of them.
  \item [(B2)] If $\Sigma_1$ and $\Sigma_2$ are two apartments
    containing two cells $c_1,c_2$ of $X$, then there exists a isometry $f: \Sigma_1 \to \Sigma_2$ which fixes
    $c_1$ and $c_2$ pointwise.
  \item [(B3)] Each apartment $\Sigma$ is isometric
    to the planar tessellation ${C(W,S)}$.
  \end{itemize}
  The building $X$ is called {\em Euclidean} if
  ${\mathbb X} = \R^2$ and {\em hyperbolic} if ${\mathbb X} = {\mathbb H}^2$. A building is
  called {\em thick} if every edge is contained in at least three chambers.
  A building which is not thick is called a {\em thin} building.
\end{dfn}

\begin{prop}\label{p:Buildings=PCPS}
Every $2$-dimensional Euclidean or hyperbolic building is a {\PCPS}, i.e., it satisfies the axioms (PCPS1), (PCPS2), (PCPS3).
\end{prop}
\begin{proof}
Disregarding the additional Euclidean or hyperbolic structure of the cells of a building, we can view it and its apartments as polygonal complexes in the sense of Definitions \ref{dfn:polycomp} and \ref{dfn:plantess}. Since the apartments of buildings are always convex (see \cite[p.~164, l.~-5]{GP} and also \cite[Corollary~5.54]{AB} or \cite[Proposition on p.~59]{Ga} for simplicial buildings), we see that every building is a {\PCPS}.
\end{proof}
\subsubsection{Euclidean buildings}\label{s:Euclidean buildings}
In this subsection we discuss how our theory applies to
 Euclidean buildings and give two specific examples.

As discussed above the Coxeter polygon $P$ has to be a $k$-gon whose interior angles are given by $\pi/m_{1},\ldots,\pi/m_{k}$ with integers $m_{1},\dots,m_{k}\ge2$ which have to satisfy
\begin{align*}
    (k-2)\pi=\frac{\pi}{m_{1}}+\ldots+\frac{\pi}{m_{k}}
\end{align*}
due to the Euclidean structure. This implies $k\leq 4$. As for $P$ being a triangle, $k=3$, one has either of the interior angles $\{
\frac{\pi}{3},\frac{\pi}{3},\frac{\pi}{3} \}$, $\{
\frac{\pi}{2},\frac{\pi}{4},\frac{\pi}{4} \}$ or $\{
\frac{\pi}{2},\frac{\pi}{3},\frac{\pi}{6} \}$.
Each of these choices leads to a unique Coxeter group and to a class of Euclidean buildings which are said to be of
type $\widetilde A_2$, $\widetilde C_2$ and $\widetilde G_2$, respectively, see \cite[Example~10.14]{AB}. For $k=4$ the only possibility for $P$ is to be the regular equilateral, the square.

By this discussion the following proposition can be checked immediately. We highlight it as it clarifies the applicability of the results of the previous sections to Euclidean buildings.

\begin{prop} \label{prop:curvEuclid}
For every $2$-dimensional Euclidian building, we have $\ka^{(\Sigma)}=0$, for every apartment $\Sigma$. Moreover, the sectional corner curvature $\ka_{c}^{(\Sigma)}$ is constantly zero on every apartment $\Sigma$ if and only if the Coxeter polygon is an equilateral triangle (type $\widetilde A_{2}$) or a square. Otherwise, some of the sectional corner curvatures are strictly positive.
\end{prop}
\begin{proof}For the Coxeter polygon $P$ with interior angles are given by $\pi/m_{1}, \ldots,$ $\pi/m_{k}$, the vertex degrees in the apartments of corresponding buildings have to be $2m_{1},\ldots,2m_{k}$ in order to sum up to $2\pi$ about each vertex. This gives the result by  direct calculation.
\end{proof}

Let us stress that even though there are only three types of Euclidean triangles as Coxeter polygons, a classification of all buildings of one of these types is impossible because of their abundance (see \cite[p. 157]{Ro}).

Next we focus on two examples in more detail. First we revisit Example~\ref{ex:vdov} in Subsection \ref{s:PCPS} in more detail.

 \medskip

\noindent\textbf{Example
\ref{ex:vdov} (revisited)} This example   is a thick  Euclidean building based on an equilateral Euclidean triangle. Thus, it is   of type $\widetilde A_2$ and has, therefore, zero sectional corner curvature.

To get a better understanding of this building, it is worth looking at  the links of its vertices. It can be checked, that these links are all isomorphic to the Heawood graph which is a generalized 3-gon.
\medskip

Next, we consider a natural class of Euclidean buildings based on a square.

\begin{example}[Product of trees] \label{ex:prodtrees} Let $r,s \ge 2$
  and $T_r$ and $T_s$ be infinite regular metric trees of vertex
  degrees $r$ and $s$, respectively. All edge lengths are chosen to be
  $1$. We can think of one of the trees, say $T_r$, to be {\em
    horizontal} and the other one to be {\em vertical}. Then the
  product $T_r \times T_s$ carries a natural structure of a thick
  Euclidean building $X=(V,E,F)$ with $P = [0,1]^2 \subset {\mathbb
    R}^2$. The set $V$ consists of all pairs $(x,y)$ where $x$ and $y$
  are vertices in $T_r$ and $T_s$ respectively. Two vertices
  $(x_1,y_1), (x_2,y_2) \in V$ are connected by an edge in $E$, if
  either ($x_1 = x_2$ and $y_1 \sim_{T_s} y_2$) or ($y_1 = y_2$ and
  $x_1 \sim_{T_r} x_2$). In the first case we call the edge in $E$
  horizontal and in the second case we call the edge in $E$
  vertical. The chambers are the unit squares with boundary vertices
  $(x_1,y_1), (x_1,y_2), (x_2,y_1), (x_2,y_2)$ for any choice $x_1
  \sim_{T_r} x_2$ and $y_1 \sim_{T_s} y_2$. All vertices in $T_r
  \times T_s$ have degree $r+s$ (with $r$ emanating horizontal and $s$
  emanating vertices edges). Moreover, a vertical edge is contained in
  precisely $r$ chambers while a horizontal edge is contained in
  precisely $s$ chambers.

  Two bi-infinite combinatorial geodesics $g_1
  \subset T_r$ and $g_2 \subset T_s$ can be viewed as infinite regular
  trees of vertex degrees $2$.
  The corresponding subcomplex $\Sigma
  = \Sigma_{g_1,g_2} = g_1 \times g_2$ is isomorphic to a regular
  tessellation of $\R^2$ by unit squares. We choose ${\mathcal A}$ to be the set of all those subcomplexes.

  From  the proposition above we learn that the sectional corner curvatures are constantly zero, i. e., $\ka^{(\Sigma)}_{c}=0$ for every apartment.

  Another interesting fact about these buildings is that the link of every vertex in $T_r \times T_2$ is the complete   bipartite graph $K_{r,s}$.
\end{example}

\subsubsection{Hyperbolic buildings}
Finally, let us consider some examples of hyperbolic buildings.

In the hyperbolic case, the Coxeter polygon $P$ has to be a $k$-gon whose interior angles $\pi/m_{1},\ldots,\pi/m_{k}$ with integers $m_{1},\dots,m_{k}\ge2$ have to satisfy
\begin{align*}
    (k-2)\pi>\frac{\pi}{m_{1}}+\ldots+\frac{\pi}{m_{k}}
\end{align*}
due to the hyperbolic structure.

This gives the following immediate consequence.

\begin{prop}For every $2$-dimensional hyperbolic building, we have $\ka^{(\Sigma)}<0$, for every apartment $\Sigma$. Moreover, the sectional corner curvature satisfies $\ka_{c}^{(\Sigma)}<0$  if the Coxeter polygon is a regular hyperbolic polygon.
\end{prop}
\begin{proof}Again, the vertex degrees in the apartments of corresponding buildings have to be $2m_{1},\ldots,2m_{k}$ in order to sum up to $2\pi$ about each vertex. This gives the result by  direct calculation using the discussion above.
\end{proof}

Note that while all hyperbolic buildings have negative sectional face
curvature they do not always have also non-positive sectional corner
curvature: consider a tessellation of the hyperbolic plane by
triangles with interior angles
$\frac{\pi}{r},\frac{\pi}{s},\frac{\pi}{t}$ with $r,s,t \ge 2$ and
$ \frac{1}{r} + \frac{1}{s} + \frac{1}{t} < 1 $ (which has to be satisfied as the sum over the angles of a hyperbolic triangle has to be less than $\pi$).
This tessellation is a thin hyperbolic building and it has
non-positive corner curvature if and only if $r,s,t \ge
3$.

Henceforth, we only consider hyperbolic buildings with {\em regular   polygons} as faces. These hyperbolic buildings have always negative sectional corner curvature by the above proposition.

Below, we briefly outline three examples of hyperbolic buildings and refer the interested readers to the corresponding references.

We start with hyperbolic buildings whose faces are right-angled polygons.

\begin{example}[``Bourdon buildings''] \label{ex:bourdon} Let $p \ge
  5$ and $q \ge 3$. Then there is a {\em unique} hyperbolic building
  $X_{p,q}$ with the following properties (see \cite{Bou}): All
  chambers are regular right-angled hyperbolic $p$-gons and the link
  $L(v)$ of every vertex is the complete bipartite graph
  $K_{q,q}$. Since every edge of $X_{p,q}$ lies in $q$ chambers,
  $X_{p,q}$ is a thick building. Moreover, $X_{p,q}$ has constant negative  sectional corner
  curvature $\kappa_c^{(\Sigma)} = 1/p - 1/4 < 0$.
\end{example}
\bigskip

Next, we mention a general method to obtain hyperbolic buildings admitting a cocompact group action. First, we choose finitely
many hyperbolic polygons, label their oriented edges and identify edges with the same labels (these edges must obviously have the same length). We call such a compact polygonal complex a {\em polyhedron}. Then its universal covering is again a polygonal complex (admitting a cocompact group action with this polyhedron as its quotient) and the links of its vertices provide useful
information in the decision whether it is a building (see,  e.g., \cite{GP}).

Next, we give an example which uses this construction.

\begin{example}[see \cite{Vd,KVd}] \label{ex:vdov2} Let ${K}$
  be a {\em polygonal presentation} associated to the disjoint
  connected bipartite graphs $G_1,\dots,G_n$ in the sense of
  \cite[Definition~1.2]{KVd}. Assume that all $G_i$ are copies of the same
  generalized $m$-gon. Every cyclic $p$-tuple in ${K}$ provides a clockwise labeling of the oriented edges of
  a regular hyperbolic $p$-gon with angles $\frac{\pi}{m}$. If $mp >
  2m+p$ then the universal covering of the polyhedron corresponding to
  ${K}$ is a hyperbolic building, see \cite[p. 472]{Vd}. It has constant sectional corner curvature $\ka_{c}^{(\Sigma)}=(2m+p-mp)/(2mp)<0$.
   This approach provides
  examples of hyperbolic buildings with $p$-sided chambers for
  arbitrary $p \ge 3$ {\em with a cocompact group action}. \\In particular, the
  triangle presentations given in \cite{KVd} lead to explicit
  hyperbolic buildings with regular triangles as faces.
\end{example}

Finally, techniques of Haglund \cite{Hag} provide us with the following result.

\begin{example}[see {\cite[Thme. 3.6]{GP}}] \label{ex:hag} Let $P \subset {\mathbb H}^2$ be a regular hyperbolic polygon with angles $\frac{\pi}{m}$, $m \ge 3$ and an even number of sides. Let $(W,S)$ be   the associated Coxeter group. Let $L$ be an algebraic generalized   $m$-gon over a field with large enough cardinality. (The term     ``algebraic'' refers to the fact that the $m$-gon is based on a     Chevalley quadruple, see \cite[Definition~3.3]{GP}.) Then there are {\em     uncountably many} hyperbolic buildings of type $(W,S)$ with faces   isometric to $P$ such that all links are isomorphic to  $L$.
\end{example}


\subsection{Maximal apartment systems in buildings}

Since any union of apartment systems of a building $X=(V,E,F)$ forms again an apartment system (see \cite[Thm. 4.54]{AB}) for a proof in the case of simplicial buildings), there exists a unique maximal system of apartments by Zorn's lemma. In the proof of positive Cheeger constant as a consequence of negative curvature (Theorem~\ref{t:cheeger1}), we used the stronger axiom (${\rm PCPS1}^*$) instead of (PCPS1). Below we show that, for a building with maximal apartment system, (${\rm PCPS1}^*$) is satisfied. We give the full reference for the simplicial case and we believe that the result remains true in the polygonal case as well. The proof was indicated to us by Shahar Mozes.

\begin{thm} \label{thm:maxbuildPCPS}
Every locally finite $2$-dimensional Euclidean or hyperbolic building with a maximal apartment system   satisfies the axioms (${\it PCPS1}^*$), (PCPS2), (PCPS3).
\end{thm}
\begin{proof}  By Proposition~\ref{p:Buildings=PCPS} we only have to show   (PCPS1$^*$), that is, every one-sided infinite geodesic is included   in an apartment. Consider a one-sided infinite geodesic $(f_j)_{j \ge 0}$ of faces. Define $\mathcal{A}_{0}$ to be the set of all apartments that contain $f_{0}$. Define a metric $\delta$ on $\mathcal{A}_{0}$ viz
\begin{align*}
    \delta(\Sigma_{1},\Sigma_{2})=1/\max\{r\in\N\mid \Sigma_{1}\cap B_{r}(f_{0})=\Sigma_{2}\cap B_{r}(f_{0})\}
\end{align*}
for $\Sigma_{1}\neq\Sigma_{2}$ and $0$, otherwise. We show that the metric space $(\mathcal{A}_{0},\delta)$ is compact by showing that is totally bounded and complete. Note that total boundedness of $(\mathcal{A}_{0},\delta)$, (i. e., the metric space can be covered by finitely many $\eps$ balls for every $\eps>0$) follows from local finiteness, as local finiteness  implies the set $\{\Sigma\cap B_{r}(f_{0})\mid \Sigma\in\mathcal{A}_{0}\}$ is finite for all $r$. In order to see completeness, we let $(\Sigma_{n})$ be a Cauchy sequence in $\mathcal{A}_{0}$ and observe that, for a given $r$, there is $N$ such that $b_{r}=\Sigma_{n}\cap B_{r}(f_{0})$ are constant for $n\ge N$. One can check that $\Sigma=\bigcup_{r\ge1} b_{r}$ is isometric to the Coxeter complex $C(W,S)$ and, thus, $\Sigma$ is contained in the system of maximal apartments by \cite[Proposition~4.59]{AB}. Hence, $\Sigma\in \mathcal{A}_{0}$ and, thus, $\Sigma $ is a limit of $(\Sigma_{n})$ in $\mathcal{A}_{0}$. Hence, $(\mathcal{A}_{0},\delta)$ is totally bounded and complete and, thus, compact. Now, let $\Sigma_{n}\in\mathcal{A}_{0}$ be an apartment that contains $f_{n}$ and, by convexity of the apartments, $f_{0},\ldots,f_{n}\in \Sigma_{n}$. By compactness, there is a convergent subsequence with limit $\Sigma\in \mathcal{A}_{0}$ which therefore contains the faces of the geodesic $(f_{j})_{j\ge0}$.
\end{proof}

Let us close this section by a  one-dimensional example that shows that the choice of the apartment system is not unique.
Analogues in higher dimensions are easy to find.

\begin{example} Let $T_r = (V,E)$ be a regular metric tree of edge
  length $1$ and vertex degree $r \ge 3$, and let $\phi: E \to \{
  1,2,\dots,r \}$ be a labeling of the edges such that the $r$ edges
  emanating from every vertex carry pairwise different labels. Let
  ${\mathcal A}$ be the set of bi-infinite paths $( f_k )$ such
  that the bi-infinite sequence $x_k = \phi(f_k)$ has no doublings
  (i.e., $x_k \neq x_{k+1}$ for all $k \in \Z$) and is periodic (i.e.,
  there exists $t \ge 1$ such that $x_{k+t} = x_k$ for all $k \in   \Z$). Then it is easy to see that $T_r$ together with ${\mathcal A}$
  as its system of apartments forms a one-dimensional Euclidean
  building. Another choice ${\mathcal A}'$ of an apartment system is
  the set of all bi-infinite paths without doublings in the above sense, which is the maximal apartment system. It is obvious that ${\mathcal A}'$ is a
  strictly bigger apartment system than ${\mathcal A}$.
\end{example}


\end{document}